\documentclass[12pt]{amsart}

\usepackage{amsmath}
\usepackage{amscd}
\usepackage{amssymb}
\usepackage{latexsym}
\usepackage{color}
\usepackage{accents}

\newtheorem{theorem}{Theorem}[section]
\newtheorem*{theorem*}{Theorem}
\newtheorem{lemma}[theorem]{Lemma}
\newtheorem*{lemma*}{Lemma}
\newtheorem{corollary}[theorem]{Corollary}
\newtheorem*{corollary*}{Corollary}
\newtheorem{proposition}[theorem]{Proposition}

\newtheorem{remark}[theorem]{Remark}
\newtheorem{question}[theorem]{Question}
\newtheorem{definition}[theorem]{Definition}

\newtheorem{example}[theorem]{Example}
\newcommand{\bgl}{\begin{equation}} 
\newcommand{\egl}{\end{equation}}
\newcommand{\bgloz}{\begin{equation*}} 
\newcommand{\egloz}{\end{equation*}}
\newcommand{\bgln}{\begin{eqnarray}} 
\newcommand{\egln}{\end{eqnarray}}
\newcommand{\bglnoz}{\begin{eqnarray*}} 
\newcommand{\eglnoz}{\end{eqnarray*}}
\newcommand{\btheo}{\begin{theorem}}
\newcommand{\etheo}{\end{theorem}}
\newcommand{\btheooz}{\begin{theorem*}}
\newcommand{\etheooz}{\end{theorem*}}
\newcommand{\blemma}{\begin{lemma}}
\newcommand{\elemma}{\end{lemma}}
\newcommand{\blemmaoz}{\begin{lemma*}}
\newcommand{\elemmaoz}{\end{lemma*}}
\newcommand{\bproof}{\begin{proof}}
\newcommand{\eproof}{\end{proof}}
\newcommand{\bbew}{\begin{beweis}}
\newcommand{\ebew}{\end{beweis}}
\newcommand{\bremark}{\begin{remark}\em}
\newcommand{\eremark}{\end{remark}}
\newcommand{\bquestion}{\begin{question}\em}
\newcommand{\equestion}{\end{question}}
\newcommand{\bdefin}{\begin{definition}}
\newcommand{\edefin}{\end{definition}}
\newcommand{\bprop}{\begin{proposition}}
\newcommand{\eprop}{\end{proposition}}
\newcommand{\bcor}{\begin{corollary}}
\newcommand{\ecor}{\end{corollary}}
\newcommand{\bcoroz}{\begin{corollary*}}
\newcommand{\ecoroz}{\end{corollary*}}
\newcommand{\bfa}{\begin{cases}} 
\newcommand{\efa}{\end{cases}}
\newcommand{\bexample}{\begin{example}\em}
\newcommand{\eexample}{\end{example}}
\newcommand{\cF}{\mathcal F}
\newcommand{\cG}{\mathcal G}

\def\Cz{\mathbb{C}}

\def\Nz{\mathbb{N}}

\def\Rz{\mathbb{R}}

\def\1z{\mathbb{1}}
\newcommand{\bk}{{\bf k}}
\newcommand{\bfs}{{\bf s}}
\newcommand{\bft}{{\bf t}}
\newcommand{\lori}{\longrightarrow}

\newcommand{\ma}{\mapsto} 
\newcommand{\Rarr}{\Rightarrow} 

\newcommand{\ve}{\varepsilon}

\def\SEMI{\mbox{$\times\kern-2pt\vrule height5pt width.6pt \kern3pt $}}

\newcommand{\abs}[1]{\left|#1\right|} 
\newcommand{\norm}[1]{\left\|#1\right\|} 
\newcommand{\defeq}{\mathrel{:=}} 

\newcommand{\dop}{\text{: }} 

\newcommand{\supp}{{\rm supp}\,}
\newcommand{\suppc}{{\rm supp}}
\newcommand{\cGu}{\cG^{(0)}}
\newcommand{\lge}{\left\{} 
\newcommand{\rge}{\right\}} 
\newcommand{\lru}{\left(} 
\newcommand{\rru}{\right)} 
\newcommand{\lsp}{\left\langle} 
\newcommand{\rsp}{\right\rangle} 
\newcommand{\rukl}[1]{\lru #1 \rru} 
\newcommand{\gekl}[1]{\lge #1 \rge} 
\newcommand{\spkl}[1]{\lsp #1 \rsp} 
\newcommand{\vp}{\varphi}
\newcommand{\menge}[2]{\gekl{ #1 \dop #2 }} 
\begin{document}

\title[]{\footnotesize  Amenability, Reiter's condition and Liouville property}
\author{Cho-Ho Chu and Xin Li}

\date{}

\begin{abstract}
We show that the Liouville property and Reiter's condition are equivalent for semigroupoids. This result applies to
semigroups as well as semigroup actions. In the special case of  measured groupoids and locally compact groupoids, our result
 proves
Kaimanovich's conjecture of the equivalence of amenability and the Liouville property.
\end{abstract}

\address{School of Mathematical Sciences,
  Queen Mary, University of London,
 London E1 4NS, UK}

 \email{ c.chu@qmul.ac.uk, xin.li@qmul.ac.uk}


\maketitle



\section{Introduction}

Since the seminal work of von Neumann \cite{vn}, amenable groups and semigroups have had a profound impact on many areas of mathematics. Amenability of locally compact groups has been shown to be equivalent to many fundamental properties
in harmonic analysis including the Liouville property which is one of the subjects of the present paper.
In operator algebras, amenability plays a pivotal  role in  their classification (cf.\cite{{Con,ell}}) as well as in the recent progress on the Novikov conjecture (see \cite{HK,Oz}). Indeed, amenable groupoids satisfy the Baum-Connes conjecture \cite{tu}. Amenability also plays a significant role in the recent development of semigroup C*-algebras relating to some aspects of number theory \cite{CDL,Li,Li2}.

A locally compact group $G$ is amenable if there is a left invariant mean on $L^\infty(G)$. A topological semigroup $G$ is usually called {\it amenable} if there is a left invariant mean on the algebra $LUC(G)$ of bounded left uniformly continuous functions on $G$. These two definitions of amenability are equivalent for locally compact groups. The notion of amenability has been extended to group actions by Zimmer \cite{{Z77,Z78}}. For the more
general case of groupoids which, among other things, unify both concepts of groups and group actions, it was introduced by Renault \cite{{R,AR}}. Amenable groupoids were defined in terms of Reiter's condition, which stipulates the existence of nets of approximately invariant probability measures and was first formulated by Day \cite{day}
 for discrete semigroups. For locally compact groups, Reiter's condition is equivalent to amenability and therefore the definition of an amenable groupoid is a natural extension of the group case. However, for topological semigroups, the question of whether Reiter's condition follows from amenability as defined previously appear to be open \cite[p.321]{Lau}.

The equivalence of amenability and the Liouville property for $\sigma$-compact locally compact groups was first conjectured by Furstenberg \cite{F} and proved by Rosenblatt \cite{r}, Kaimanovich and Vershik \cite{KV}.
More recently, Kaimanovich introduced the Liouville property for groupoids in \cite{K2} and conjectured its equivalence to amenability, having proved that the former implies the latter. For semigroups, the Liouville property for abelian semigroups has been  studied in \cite{ds,lz,RS}, but its connection to amenability has not been the subject of investigation before.

Our main objective in this paper is  to clarify  the relationships of amenability, Reiter's condition and the Liouville property in the setting of semigroupoids, which subsumes and provides a unified treatment to the important cases of groupoids, semigroups and transformation semigroups.
We introduce and study the Liouville property  and Reiter's condition for semigroupoids. We  prove that a semigroupoid possesses the Liouville property if and only if it satisfies Reiter's condition (Theorems~\ref{L-->R_m}, \ref{L-->R_t}, \ref{R-->L_m}, \ref{R-->L_t}). An immediate consequence is the equivalence of the Liouville property and Reiter's condition for semigroup actions
(Theorem~\ref{87}) as well as the equivalence of amenability and the Liouville property for discrete semigroups (Theorem~\ref{610}) and also, for both measured groupoids and topological groupoids (Theorems~\ref{mgpd_LR}, \ref{tgpd_LR}), the latter proves a conjecture of Kaimanovich in \cite{K2}. We thank Vadim Kaimanovich for informing us, after we have written this paper, that his conjecture for {\it measured groupoids} has also been proved by Theo B\"uhler and himself in an unpublished note. Our result includes the case of topological groupoids, which requires some refinements of Reiter's condition.

A Riemannian manifold is said to have the Liouville property if it does not admit non-constant bounded harmonic functions. Examples include complete manifolds with non-negative Ricci curvature, by a well-known result of Yau \cite{yau}.
To introduce the Liouville property for groups, we begin with  a connected Lie group $G$. The Laplace operator of $G$ generates a one-parameter convolution semigroup $(\pi_t)_{t>0}$ of probability measures on $G$ such that a function $f\in L^\infty(G)$ is $C^2$ and harmonic if and only if it satisfies the convolution equations $f=f * \pi_t$ for all $t>0$ (cf. \cite{hunt} and \cite[Proposition~V.6]{a}). This is equivalent to $f= f*\pi_t$ for {\it some} $t>0$ if $G$ is of type T \cite[p.136]{a} (e.g. $G$ is semisimple with finite centre \cite[Th\'eor\`eme II.1]{a}).
  More generally, given a probability measure $\pi$ on a locally compact group $G$, a Borel function $f : G \rightarrow \mathbb{C}$ is called {\it $\pi$-harmonic} if it satisfies the convolution equation $f = f*\pi$.
The latter condition is an analogue of the mean value property which characterises harmonic functions on manifolds.
  We say that $G$ has the {\it Liouville property} if there is an absolutely continuous probability measure $\pi$ on $G$ such that all bounded $\pi$-harmonic functions are constant.
 In \S\,\ref{sgpd}, we extend the definition of Liouville property to semigroupoids, which are algebraic structures that resemble a semigroup, except that multiplication is not globally defined.

We begin our discussion of semigroupoids in the next section, followed by an introduction to  Reiter's condition and the Liouville property, in both the measurable and topological contexts.
To pave our way, we prove some basic results  in \S~\ref{sgpd1} concerning convolution of systems of measures
on semigroupoids.
 In \S\,\ref{LR_sgpd}, we show that the Liouville property implies Reiter's condition for semigroupoids, both in the measurable and topological setting. This extends the result in \cite{K2} for groupoids equipped with a Haar system.
The converse of the previous result for semigroupoids is proved in \S\,\ref{RL_sgpd}. In both results, we only make the weaker assumption of a {\it quasi-Haar system} for  semigroupoids so that wider applicability can be achieved. For example, such semigroupoids include discrete semigroups, which need not admit a Haar system.
The following section (\S~\ref{special}) is devoted to the special cases of
groupoids, transformation semigroups and  semigroups.
Our results on semigroupoids apply directly to semigroups equipped with a {\it quasi-invariant measure} (Theorem~\ref{special_thm_meas-sgp}).
For semigroups without such a measure and not necessarily locally compact,
we discuss briefly  the case of metrizable semigroups, for which one can show that Reiter's condition implies the Liouville property, which in turn implies amenability. The proof of the last implication is different from that for semigroupoids. However, a more thorough treatment will be given in another work \cite{CLi}. We conclude the paper with some examples of semigroups with the Liouville property in \S\,\ref{nilp}.

\section{Semigroupoids}
\label{sgpd}

 We initiate the discussion of semigroupoids with semigroups. By a {\it topological semigroup}, we mean a semigroup $S$ endowed with a topology such that the multiplication on $S$ is jointly continuous. If moreover, the topology is metrizable, we
call $S$ a {\it metric semigroup}.

To discuss semigroupoids in measurable and topological settings, we need to fix some notations for measures and functions on a topological space $X$. A positive finite Borel measure $\mu$ on $X$ is called {\it tight} if for each $\varepsilon > 0$, there is a compact set $K \subseteq X$ such that $\mu  (X \setminus K)<\varepsilon$. Every positive finite Borel measure on a {\it complete separable} metric space is tight.

Let $M(X)$ be the space of complex-valued regular Borel measures on $X$. For each nonzero $\mu \in M(X)$, the norm of $\mu $ is defined by $\Vert \mu \Vert = \vert \mu  \vert  (X)$ where $|\mu|$ is the total variation of $\mu$. The {\it support} of $|\mu|$ is defined by
\[{\rm supp}\, |\mu| = \bigcap\{F\subseteq X: F ~~{\rm is~~ closed}, ~|\mu|(F) = |\mu|(X)\}\]
and the {\it support} of $\mu$, denoted by supp\,$\mu$, is defined to be that of $|\mu|$.   A measure $\mu  \in M(X)$ is called {\it tight} if $\vert  \mu   \vert  $ is tight. Every Borel measure with compact support is a tight measure. We denote by $M_t(X)$ the subspace of $M(X)$ consisting of tight measures on $X$. Let $C_b(X)$ be the $C^*$-algebra of bounded complex continuous functions on $X$. For a metric space $X$, the support {supp}\,$\mu$ of each tight measure $\mu$ is nonempty and separable. Moreover, for every Borel set $B \subseteq X$ and $\varepsilon > 0,$ there is a compact set $K \subseteq B$ satisfying $|\mu|  (B \setminus K) < \varepsilon$.

 We recall that
a topological space $Y$ is a {\it Polish space} if it is homeomorphic to a complete separable metric space. A subset
of $Y$ is called {\it analytic} if it is of form $f(Z)$ for some continuous function $f$ from a Polish space $Z$ to $Y$. Every Borel set in a Polish space is analytic.

In the sequel, by  an {\it analytic} Borel space, we mean a measurable space $( \mathcal{G}, \mathcal{B})$ which is isomorphic
to an analytic set in a Polish space with the relative Borel structure. The sets in the $\sigma$-algebra $\mathcal{B}$
are called the {\it Borel sets} in $\mathcal{G}$. The $\sigma$-algebra $\mathcal{B}$ is often not written
explicitly for an analytic Borel space $\mathcal{G}$. We refer to \cite{av} for the basic properties of analytic Borel spaces.

\begin{remark}\label{sep}\rm
Since an uncountable Polish space with its Borel structure is isomorphic to the unit interval $[0,1]$
with the usual Borel structure \cite[p.\,451]{ku}, we see that the Banach space $L^1(Y, \mu)$ of $\mu$-integrable functions on $Y$ is separable for any Borel measure $\mu$ on a Polish space $Y$. The same is true for $L^1(\cG, \nu)$, where $\nu$ is a Borel measure on an analytic Borel space $\cG$.
\end{remark}

 Let $S$ be a topological semigroup. A function $f \in C_b(S)$ is called {\it left uniformly continuous} if the mapping $a\in S \mapsto \delta_a * f \in C_b(S)$ is continuous, where $\delta_a *f$ is the left translate of $f$ by $a$, which is
 the convolution of the point mass $\delta_a$ and $f$ defined below. The space of bounded left uniformly continuous functions on $S$ will be denoted by $LUC(S)$ which forms a sub-C*-algebra of $C_b(S)$. For a discrete semigroup $S$, we have $LUC(S) = \ell^\infty(S)$.

Given $\pi, \sigma \in M(S)$, we define their convolution $\pi * \sigma$ to be the image of product measure $\pi \times \sigma$
under the map $(x,y) \in S\times S \mapsto xy\in S$, that is,
$$(\pi * \sigma) (E) = \int_{S\times S} \chi_E (xy)d(\pi \times \sigma)(x,y)$$ for each Borel set $E \subseteq S$. For $f\in C_b(S)$, we have
$$\int_S f d(\pi * \sigma) = \int_S \int_S f(xy) d\pi(x) d\sigma(y).$$
In the sequel, we denote by $\pi^n$ the $n$-fold convolution $\overbrace{\pi * \cdots * \pi}^{\rm n-times}$
of $\pi$.

Given $\pi \in M(S)$ and a Borel function $f: S\rightarrow \mathbb{C}$ , we define the {\it semigroup convolutions} $f* \pi$ and $\pi * f$ by $$(f * \pi) (x) = \int_S f(xy) d\pi(y) \quad {\rm and} \quad (\pi * f) (x) = \int_S f(yx)d\pi(y)$$
if the integrals exist.
 For a point mass $\delta_a$ at $a\in S$,  we have $(\delta_a * f) (x) = f(ax)$ and $(f * \delta_a) (x)= f(xa)$ as well as
$$\int_S (\delta_a * f) d \pi = \int_S f d(\delta_a * \pi).$$

\begin{definition}\label{ph}\rm Let $S$ be a topological semigroup and let $\pi \in M(S)$ be a probability measure.  A Borel function $f : S \rightarrow \mathbb{C}$ is called {\it $\pi$-harmonic} if $f*\pi =f$.
\end{definition}


\begin{definition}\rm
A topological semigroup $S$ is called
({\it left}) {\it amenable} if there is a left invariant mean on $LUC(S)$, that is, there exists a norm-one positive linear functional $\vp : LUC(S) \rightarrow \mathbb{C}$ on the C*-algebra $LUC(S)$ satisfying $\vp (\delta_a *f) = \vp (f)$ for all $a \in S$.
\end{definition}
For a locally compact group $S$, this definition of amenability agrees with the usual one \cite[p.67]{nam}.
We now turn to semigroupoids, which
generalise semigroups  and groupoids at the same time, in a very natural way.
A {\it semigroupoid} is a small category. It consists of a set $\cGu$ of objects (called units), a set $\cG$ of morphisms, the surjective {\it source} and {\it target} maps $\bfs, \, \bft : \: \cG \to \cGu$, and a composition map $ (\zeta, \eta)\in \cG^{(2)} \ma \zeta \eta \in \cG \,   $ on $\cG^{(2)} \defeq \menge{(\zeta, \eta) \in \cG^2}{\bfs(\zeta) = \bft(\eta)}$. In the special case where there is only one unit, the concept of semigroupoids reduces to semigroups with identity, which are sometimes called \textit{monoids}. As usual, we identify the units with the corresponding identity morphisms in $\cG$ (which exist by definition). In this way, we may consider $\cGu$ as a subset of $\cG$ and by a slight abuse of language, we call $\cG$ a {\it semigroupoid}.

For each $x \in \cGu$, we write $\cG^x \defeq \bft^{-1}(x) = \menge{\gamma \in \cG}{\bft(\gamma) = x}$. One may view the target map $\bft : \: \cG = \bigcup_{x\in \cGu} \cG^x\to \cGu$ like  a {\it `bundle projection'}. Actually, we will see that the Liouville property for $\cG$ is defined on the {\it `fibres'} $\cG^x$.

A semigroupoid $\cG$  is called {\it Borel} if $\cG$ is an analytic Borel space, $\cGu$ is an analytic Borel space (as a subspace of $\cG$) such that $\bfs$, $\bft$ and the composition map are Borel. In this case, the sets $\cG^{x}$ and the maps $\zeta : \: \eta \in \cG^{\bfs(\zeta)} \mapsto \zeta\eta \in \cG^{\bft(\zeta)}$ are Borel. We call $\cG$ a {\it topological semigroupoid} if it is endowed with a topology compatible with the semigroupoid structure, that is, the composition map, $\bfs$ and $\bft$ are continuous. We denote by $C_c(\cG)$ the space of complex continuous
functions on $\cG$ with compact support.

By a {\it locally compact semigroupoid}, we mean a topological semigroupoid which is locally compact Hausdorff and second countable. This implies that all our locally compact semigroupoids are $\sigma$-compact.

\bexample
Groupoids, which are by definition small categories in which every  morphism is invertible, form a special class of semigroupoids. We refer  to \cite{AR,R} for more information about groupoids.

Another class of examples is given by semigroups, or more generally, semigroup actions. Let $S$ be a semigroup with identity $e$ acting on a set $X$ from the right. We denote the action by $(x,s) \in X \times S \ma x.s \in X$, where $x.e =x$. The semigroupoid  attached to the transformation semigroup $X \curvearrowleft S$ is given by $\mathcal{G} := X \rtimes S \defeq X \times S$ (as a set), with the source and target maps $\bfs: \: (x,s) \in X \rtimes S \ma x.s \in X $, $\bft: \: (x,s) \in X \rtimes S \ma x\in X $, and the composition $(x,s)(x.s,t) = (x,st)$. It is clear that the set $\cG^{(0)} = \menge{(x,e)}{x \in X}$ of units of $X \rtimes S$ can be canonically identified with $X$. By taking $X = \gekl{\rm pt}$, one can also view the semigroup $S$ itself as the semigroupoid $ \gekl{\rm pt} \rtimes S$.

If $X$ and $S$ are  analytic Borel spaces, or topological spaces, such that $\bfs$, $\bft$ and the composition map are Borel or continuous
respectively, then $X \rtimes S$ becomes a Borel or topological semigroupoid. In particular, this is the case if $S$ is a countable discrete semigroup acting on a Borel space or  a topological space by Borel or continuous maps respectively.

Groupoids and transformation semigroups are the two motivating examples for us. We will frequently come back to them.
\eexample

We will need to consider systems of measures on the fibres $\cG^x$ of a semigroupoid $\cG$. For groupoids, these are called {\it kernels} by Connes in \cite[p.11]{Con1}.

\begin{definition}
\rm
Let $\cG$ be a Borel  semigroupoid. A {\it Borel system of measures} on $\cG$ is a family $\lambda = (\lambda^x)_{x \in \cGu}$ of $\sigma$-finite positive Borel measures $\lambda^x$ on $\cG^x$ such that for every non-negative Borel
function $f$ on $\cG$,
the map $ x \in \cGu \ma \spkl{\lambda^y,\,f}\defeq \int_{\cG} f d \lambda^x  \in [0, \infty]$ is Borel measurable, where
 we extend $\lambda^x$ naturally to a measure on $\cG$ such that $\lambda^x(\cG \setminus \cG^x) = 0$, and in particular,
 the map $x \in \cGu \ma \lambda^x(E)$ is Borel for each Borel set $E \subset \cG$.
\end{definition}

\begin{definition} \rm
Let $\cG$ be a topological semigroupoid. A {\it continuous system of measures} on $\cG$ is a family $\lambda = (\lambda^x)_{x \in \cGu}$ of (nonzero) $\sigma$-finite positive Radon measures $\lambda^x$ on $\cG^x$ such that for every $f \in C_c(\cG)$, the map $ x \in\cGu \ma \spkl{\lambda^x,f}\in \Cz$ is continuous.
\end{definition}

Given a Borel semigroupoid $\cG$ with $x \in \cGu$ and $\gamma \in \cG$ satisfying $\bfs(\gamma) = x$, and a Borel measure $\lambda^x$ on $\cG^x$, we write $\gamma \lambda^x$ for the pushforward of $\lambda^x$ under
the Borel map $\gamma: \eta \in \cG^x \ma \gamma \eta\in \cG^{\bft(\gamma)} $, which is the Borel measure
on $\cG^{\bft(\gamma)}$ induced by the map $\gamma$. By definition, we have
$$\spkl{\gamma \lambda^x,f} = \int_{\cG^x} f(\gamma \eta) d \lambda^x(\eta) $$
for each non-negative Borel function $f$ on $\cG^{\bft(\gamma)}$.

For a locally compact semigroupoid $\cG$, the pushforward $\gamma\lambda^x$ is defined analogously.
As usual, we write $\mu \prec \nu$ to mean that a measure $\mu$ is absolutely continuous with respect to $\nu$.

\begin{definition}
\label{lacm}
\rm
Let $\cG$ be a Borel (resp. topological) semigroupoid.
A Borel (resp. continuous) system of measures $\lambda = (\lambda^x)_{x \in \cGu}$ on $\cG$ is called a {\it (left) Haar system} if for all $\gamma \in \cG$, we have $\gamma \lambda^{\bfs(\gamma)} = \lambda^{\bft(\gamma)}$ and the map $(\gamma,\eta)\in \cG^2 \ma \rukl{d (\gamma \lambda^{\bfs(\gamma)}) / d \lambda^{\bft(\gamma)}} (\eta)\in [0, \infty]  $ is Borel,
where $d (\gamma \lambda^{\bfs(\gamma)}) / d \lambda^{\bft(\gamma)}$ denotes the Radon-Nikodym derivative. In the preceding
definition, the system $\lambda = (\lambda^x)_{x \in \cGu}$ will be called a {\it (left) quasi-Haar system} if the condition
$\gamma \lambda^{\bfs(\gamma)} = \lambda^{\bft(\gamma)}$ is replaced by $\gamma \lambda^{\bfs(\gamma)} \prec \lambda^{\bft(\gamma)}$.
\end{definition}

\begin{definition}
\label{proper}
\rm
A Borel system of measures $\lambda = (\lambda^x)_{x \in \cGu}$ on a Borel semigroupoid $\cG$ is called {\it proper} if there exists an increasing sequence of Borel subsets $A_n$ of $\cG$ with $\cG = \bigcup_n A_n$ such that for each $n \in \Nz$, the map $x\in \cGu \ma \lambda^x(A_n)\in \Rz$ is bounded.
\end{definition}

\begin{definition} \rm
A {\it measure semigroupoid} is a triple $(\cG,\lambda,\mu)$ consisting of a Borel semigroupoid $\cG$, a proper quasi-Haar system $\lambda = (\lambda^x)_{x \in \cGu}$ on $\cG$ and a positive Borel measure $\mu$ on $\cGu$.
\end{definition}

In the topological setting, we shall denote by $(\cG,\lambda)$ a locally compact semigroupoid  $\cG$ equipped with a quasi-Haar system $\lambda = (\lambda^x)_{x \in \cGu}$ on $\cG$.

\bremark
\label{ex_haar}
If $\lambda$ is a left Haar system on a Borel (or locally compact) groupoid, then $\lambda$ is a quasi-Haar system since $\gamma \lambda^{\bfs(\gamma)} = \lambda^{\bft(\gamma)}$. An important example of a measure semigroupoid is the  {\it measured groupoid} $(\cG,\lambda,\mu)$ studied in \cite{AR,K2} in which the Borel system $\lambda$ is assumed to be a left Haar system. Likewise, a continuous system $\lambda$ of measures on a locally compact groupoid $(\cG, \lambda)$ considered in \cite{AR,K2} is always a left Haar system. Our notion of a quasi-Haar system is a weaker version of a Haar system in the groupoid case. However, this weak version is already sufficient for our purpose and has wider applicability.

Let us now consider the case of the Borel semigroupoid $X \rtimes S$ attached to a Borel transformation semigroup $X \curvearrowleft S$.
\begin{definition}\rm
A positive Borel measure $\lambda$ on a topological semigroup $S$ is called {\it quasi-invariant} if $s\lambda \prec \lambda$ for all $s\in S$, where $s \lambda$ is the left translate of $\lambda$, which is the measure induced by the left translation $t\in S \ma st\in  S$.
\end{definition}
If $S$ is a group, this is equivalent to saying that all translates of $\lambda$ are mutually equivalent. Thus our definition extends the classical one (cf.\cite[p.58]{fo}). Let $\lambda$ be a quasi-invariant positive $\sigma$-finite Borel measure on $S$. Suppose that the map $(s,t)\in S \times S \ma \rukl{d (s \lambda) / d \lambda}(t)\in [0, \infty]$ is Borel. Then $\lambda^x \defeq \delta_x \times \lambda$ is a quasi-Haar system on $X \rtimes S$ as  $(x,s)(\delta_{x.s} \times \lambda) = \delta_x \times (s \lambda)$. Here is a concrete class of examples: Suppose that $G$ is a locally compact group, acting on an analytic Borel space $X$ with a Borel action $(x,g)\in X \times G \ma x.g\in X$.   Let $\lambda_G$ be the Haar measure on $G$. If $S$ is a Borel subsemigroup of $G$ containing the identity such that $\lambda_G(S) \neq 0$, then the restriction $\lambda$ of $\lambda_G$ to $S$ satisfies $s \lambda \prec \lambda$ for all $s \in S$ and $\rukl{d (s \lambda) / d \lambda}(t) = \chi_S(s^{-1}t)$. Therefore $\lambda^{x} \defeq \delta_{x} \times \lambda$ gives rise to a quasi-Haar system on $X \rtimes S$. In particular, if $X = \gekl{\rm pt}$, this allows us to view the semigroup $S \cong \gekl{\rm pt} \rtimes S$ as a measure semigroupoid. Of course, another class of examples is given by the case when $S$ is discrete (and countable) and $\lambda$ is the counting measure.
\eremark

\subsection{Reiter's condition}

\begin{definition} \rm Let $\cG$ be a Borel (resp. topological) semigroupoid, and $\lambda$ a quasi-Haar system on $\cG$. A Borel (resp. continuous) system $\theta = (\theta^x)_{x \in \cGu}$ of probability measures on $\cG^x$  is called {\it $\lambda$-adapted} if
\bgl
\label{abscont-RN}
  \theta^x \prec \lambda^x {\rm \ for \ every \ } x \in \cGu, \quad {\rm and ~ the~ map} \quad (\gamma,\eta)\in\cG^2 \ma {\frac{d (\gamma \theta^{\bfs(\gamma)})} {d \lambda^{\bft(\gamma)}}} (\eta) \in [0, \infty] \,  \ {\rm is \ Borel}.
\egl
\end{definition}
Note that absolute continuity in \eqref{abscont-RN} and Definition~\ref{lacm} imply that $\gamma \theta_n^{\bfs(\gamma)} \prec \gamma \lambda^{\bfs(\gamma)} \prec \lambda^{\bft(\gamma)}$ and we may consider the Radon-Nikodym derivative in \eqref{abscont-RN}.

Given a measure semigroupoid $(\cG,\lambda,\mu)$, we define a Borel measure $\mu \star\lambda  $ on $\cG$  by
$$\mu \star\lambda (E) = \int _{\cG} \lambda^x(E) d\mu(x) \in [0, \infty]$$
for each Borel set $E$ in $\cG$. For $f \in L^1(\cG, \mu \star \lambda)$, we write

$$
  \spkl{\mu\star \lambda  ,f} \defeq \int_{\cG} f d (\mu\star \lambda) =\int_{\cGu} \rukl{ \int_{\cG^x} f d \lambda^x} d \mu(x).
$$

\begin{definition} \rm A measure semigroupoid $(\cG,\lambda,\mu)$ is said to satisfy {\it Reiter's condition}, or called {\it Reiter}, if there exists a sequence $(\theta_n)_n$ of $\lambda$-adapted Borel systems of probability measures $\theta_n = (\theta_n^x)_{x \in \cGu}$ such that
\bgl
\label{Rm}
  \lim_{n \to \infty} \int_{\cG} \norm{\theta_n^{\bft(\gamma)} - \gamma \theta_n^{\bfs(\gamma)}} f(\gamma) d (\mu\star\lambda)(\gamma) = 0 \ {\rm for \ all} \ f \in L^1(\cG, \mu\star\lambda).
\egl
\end{definition}

\begin{remark}\label{8-7}\rm
We observe that $\|\theta_n^{\bft(\gamma)} - \gamma \theta_n^{\bfs(\gamma)}\|\leq 2$ and by the Lebesgue convergence theorem, the following condition implies (\ref{Rm}):
\bgl
\label{Rm'}
  \lim_{n \to \infty} \norm{\theta_n^{\bft(\gamma)} - \gamma \theta_n^{\bfs(\gamma)}} = 0 \ {\rm for} \ \mu {\rm \text{-}a.e.} \ x \in \cGu \ {\rm and} \ \lambda^x {\rm \text{-}a.e.} \ \gamma \in \cG^x.
\egl
\end{remark}
Given a continuous system of measures $\theta = (\theta^x)_{x \in \cGu}$ on a locally compact semigroupoid $\cG$ and a compact subset $K \subseteq \cGu$, we set $\suppc_K \theta \defeq \overline{\bigcup_{x \in K} \supp\theta^x}$.

\begin{definition}
\label{def-Reiter_t}
\rm
A locally compact semigroupoid $(\cG,\lambda)$ with quasi-Haar system $\lambda$ is said to satisfy {\it Reiter's condition} if there exists a sequence $(\theta_n)_n$ of $\lambda$-adapted continuous systems of probability measures $\theta_n = (\theta_n^x)_{x \in \cGu}$ such that
\bgl
\label{Rt}
  \lim_{n \to \infty} \norm{\theta_n^{\bft(\gamma)} - \gamma \theta_n^{\bfs(\gamma)}} = 0 \ {\rm for \ all} \ \gamma \in \cG.
\egl

We say that $(\cG,\lambda)$ satisfies the {\it uniform Reiter's condition} if the convergence in \eqref{Rt} is uniform on compact subsets of $\cG$.

If for every compact subset $K$ of $\cGu$, the support $\suppc_K(\theta_n)$ is compact for each $\theta_n$ in (\ref{Rt}), we say that
$(\cG,\lambda)$ satisfies the {\it tight Reiter's condition}.
\end{definition}
We also call $(\cG,\lambda)$  {\it Reiter} (respectively,  {\it uniform Reiter} or {\it tight Reiter}) if it satisfies Reiter's condition (respectively, the uniform Reiter's condition or the tight Reiter's condition).

\bremark
A measured groupoid $(\cG,\lambda,\mu)$, where $\lambda$ is a left Haar system, is Reiter in our sense if and only if it is amenable (see \cite[Chapter~3]{AR}). A locally compact groupoid $(\cG,\lambda)$, where $\lambda$ is a left Haar system, is uniform Reiter in our sense if and only if $\cG$ is amenable (see \cite[Chapter~2]{AR}). Moreover, by \cite[Theorem~2.14]{Rnew}, we know that in the case of locally compact groupoids, $(\cG,\lambda)$ is Reiter if and only if it is uniform Reiter. We will see later that for such groupoids, Reiter's condition is also equivalent to the tight Reiter's condition.

Let us now consider the case of the semigroupoid $X \rtimes S$ attached to a transformation semigroup $X \curvearrowleft S$, where $S$ is a countable discrete semigroup. For our quasi-Haar system, we always choose $\lambda^x = \delta_x \times \lambda$, where $\lambda$ is the counting measure. Going through our definitions, we see that in the Borel case, $(X \rtimes S,\lambda,\mu)$ is Reiter if there exists a sequence $(\theta_n)_n$ of families $\theta_n = (\theta_n^x)_{x \in X}$ of probability measures $\theta_n^x$ on $S$ such that for every $n \in \Nz$ and $s \in S$, the map $x\in X \ma \theta_n^x(s)\in \Rz$ is Borel and
$$\lim_{n \to \infty} \int_X \sum_{s \in S} \norm{\theta_n^x - s \theta_n^{x.s}} f(x,s) d \mu(x) = 0 \quad
\mbox{for all $f \in L^1(X \times S,\mu \times \lambda)$}.$$
In the locally compact case, our definition says that $(X \rtimes S,\lambda)$ is Reiter if there exists a sequence $(\theta_n)_n$ of families $\theta_n = (\theta_n^x)_{x \in X}$ of probability measures $\theta_n^x$ on $S$ such that for every $n \in \Nz$, the map $x\in X \ma \theta_n^x\in \ell^1(S)$ is continuous and $$\lim_{n \to \infty} \norm{\theta_n^x - s \theta_n^{x.s}} = 0 \quad
\mbox{for all $x \in X$ and $s \in S$}.$$
If the convergence is uniform on compact subsets of $X$, then $(X \rtimes S,\lambda)$ is uniform Reiter, and if   $\suppc_K \theta_n \defeq \bigcup_{x \in K} \supp\theta_n^x \subseteq S$ is finite for every compact subset $K$ of $X$, then $(X \rtimes S,\lambda)$ is tight Reiter.
\eremark

\subsection{Liouville property}
\label{defL}

Let $(\cG,\lambda,\mu)$ be a measure semigroupoid and $\pi = (\pi_x)_{x \in \cGu}$ a Borel system of probability measures. By natural extension as before, we often regard the measure $\pi_x$ on $\cG^x$ as a measure on $\cG$ vanishing on
$\cG \backslash \cG^x$. 
Given  $\pi_x \prec \lambda^x$ for all $x \in \cGu$, we can define, for every $x \in \cGu$, the {\it fibrewise Markov operator} $P_x: \: L^{\infty}(\cG^x,\lambda^x) \to L^{\infty}(\cG^x,\lambda^x)$ by
$$ P_x(f)(\zeta) = \int_{\cG^{\bfs(\zeta)}} f(\zeta \eta) d \pi_{\bfs(\zeta)}(\eta) \qquad (\zeta \in \cG^x).$$
Extending Definition \ref{ph}, a Borel function $f : \cG^x \rightarrow \mathbb{C}$ is called {\it $P_x$-harmonic} if
$P_x(f) =f$.
Let $H_{\pi,x}(\cG) \defeq \menge{f \in L^{\infty}(\cG^x,\lambda^x)}{P_x(f) = f}$ be the space of bounded $P_x$-harmonic functions on the fibre $\cG^x$.

For a locally compact semigroupid $(\cG, \lambda)$ with a $\lambda$-adapted continuous system $\pi$ of probability measures, the space $H_{\pi,x}(\cG)$ is defined analogously.

\begin{definition} \rm
A measure semigroupoid $(\cG,\lambda,\mu)$ is said to have the {\it Liouville property}, or called {\it Liouville}, if there exists a $\lambda$-adapted Borel system of probability measures $\pi = (\pi_x)_{x\in \cGu}$ such that $H_{\pi,x}(\cG) = \Cz\mathbf 1$ for $\mu$-a.e. $x \in \cGu$, where $\mathbf{1}$ is the constant function on $\cG^x$ with value $1$.
\end{definition}

\begin{definition}
\label{def-L_t}
\rm
A locally compact semigroupoid $(\cG,\lambda)$ with quasi-Haar system $\lambda$ is said to have the {\it Liouville property}, or called
 {\it Liouville}, if there exists a $\lambda$-adapted continuous system of probability measures $\pi = (\pi_x)_{x\in \cGu}$ such that $H_{\pi,x}(\cG) = \Cz\mathbf 1$ for all $x \in \cGu$.

We say that $(\cG,\lambda)$ has the {\it continuous Liouville property}, or is {\it continuous Liouville}, if  the system $\pi$ above also satisfies the condition
 that the map $(\gamma,\eta)\in\cG^2 \ma \rukl{d (\gamma \pi_{\bfs(\gamma)}) / d \lambda^{\bft(\gamma)}} (\eta)
 \in \Rz$ is bounded and continuous.

The semigroupoid $(\cG,\lambda)$ is said to have the {\it tight Liouville property} or called
 {\it tight Liouville}, if it is Liouville
 with the system $\pi$ satisfying the condition that for every compact subset $K$ of $\cGu$, $\suppc_K(\pi)$ is compact.
\end{definition}
For groupoids, these definitions coincide with the ones given in \cite{K2}.

\begin{example}\rm
In the case of the semigroupoid $X \rtimes S$ attached to a transformation semigroup $X \curvearrowleft S$, where $S$ is a discrete semigroup, the fibrewise Markov operator $P_x: \: \ell^{\infty}(S) \to \ell^{\infty}(S)$ is given by  $P_x(f)(r) = \sum_{s \in S} f(rs) \pi_{x.r}(s)$. In particular, this is just the convolution $f*\pi$ on the semigroup $S$ for $X=\{{\rm pt}\}$, in which case the Livouville property states that all bounded $\pi$-harmonic functions on $S$ are constant.
\end{example}

\section{Convolution of systems of measures}
\label{sgpd1}

In this section, we derive some basic results needed later for the convolution of two systems of measures on semigroupoids.
Let $\cG$ be a Borel semigroupoid and $\pi = (\pi_x)_{x \in \cGu}$ a Borel system of probability measures on $\cG$.
Given  a positive Borel measure $\rho$ on $\cG^x$ for some $x \in \cGu$, we first define the
 convolution $\rho * \pi$ of $\rho$ with the system $\pi$ as the following Borel measure on $\cG^x$:
$$\rho * \pi (E) = \int_{\cG^x}\pi_{\bfs(\zeta)}\{\eta\in \cG^{\bfs(\zeta)}: \zeta\eta \in E\} d\rho(\zeta)$$
for each Borel set $E$ in $\cG^x$.
For $f\in L^1(\cG^x, \rho * \pi) $, we have
$$
  \spkl{\rho * \pi,f} \defeq \int_{\cG^x} f d (\rho * \pi)
  \defeq \int_{\cG^x} \rukl{ \int_{\cG^{\bfs(\zeta)}} f(\zeta \eta) d \pi_{\bfs(\zeta)}(\eta) } d \rho(\zeta).
$$
 Now for a Borel system $\rho = (\rho_x)_{x \in \cGu}$ of measures on $\cG$, we define $(\rho * \pi)_x \defeq \rho_x * \pi$ and write $\rho * \pi = ((\rho * \pi)_x)_{x \in \cGu}$ for the family of Borel measures $(\rho * \pi)_x$ on $\cG^x$.

For a locally compact semigroupoid $\cG$ and a continuous system $\pi$ of probability measures on $\cG$, one can define analogously the convolution $\rho * \pi$ for a Borel measure $\rho$ on some $\cG^x$, and for a continuous system
$\rho$ of measures on $\cG$.

\bremark The above definition of the convolution $\rho * \pi$ coincides with the one for groupoids given
by Connes \cite[p.11]{Con1}.
\eremark


In the case of the semigroupoid $X \rtimes S$ attached to a transformation semigroup $X \curvearrowleft S$, where $S$ is a discrete semigroup, the convolution  $\rho * \pi$ is given by $(\rho * \pi)(t) = \sum_{\substack{r,s \in S \\ rs = t}} \pi_{x.r}(s) \rho(r)$.

We collect below some simple facts about convolutions of measures for later use. For a Borel semigroupoid $\cG$, we
denote by $B_b(\cG^x)$ the algebra of bounded Borel functions on a fibre $\cG^x$, equipped with the supremum norm
$\|\cdot\|_\infty$.
We recall that the total variation norm of a finite Borel measure $\rho$ on  $\cG^x$ is given by
$\|\rho\| = \sup\{|\int_{\cG^x} f d\rho| : f\in B_b(\cG^x), \|f\|_\infty \leq 1\}$.

\blemma
\label{lemma_conv-meas1} Let $\cG$ be a Borel or locally compact semigroupoid and $\pi =(\pi_x)_{x\in \cGu}$
a Borel system of probability measures on $\cG$.
\begin{itemize}
  \item[\rm a)] If $\rho$ is a finite Borel measure on $\cG^x$ for some $x \in \cGu$, then $\norm{\rho * \pi} \leq \norm{\rho}$.
  \item[\rm b)] If $\rho$ is a probability measure on $\cG^x$ ($x \in \cGu$), then $\norm{\pi_x - \rho * \pi} \leq \int_{\cG^x} \norm{\pi_x - \zeta \pi_{\bfs(\zeta)}} d \rho(\zeta)$.
\end{itemize}
\elemma
\bproof\noindent\begin{itemize}
\item[  a)] Given $f \in B_b(\cG^x)$, the map $f_{\pi}: \zeta \in \cG \mapsto \int_{\cG^{\bfs(\zeta)}} f(\zeta \eta) d \pi_{\bfs(\zeta)}(\eta)  $ is a bounded Borel function, with $\norm{f_{\pi}}_{\infty} \leq \norm{f}_{\infty}$. Hence $\abs{\spkl{\rho * \pi,f}} = \abs{\spkl{\rho,f_{\pi}}} \leq \norm{\rho} \cdot \norm{f_{\pi}}_{\infty} \leq \norm{\rho} \cdot \norm{f}_{\infty}$ and the result follows.
  \item[ b)] Let $\ve > 0$. Pick $f \in B_b(\cG^x)$ with $\norm{f}_{\infty} = 1$ and $\abs{\spkl{\pi_x - \rho * \pi,f}} \geq \norm{\pi_x - \rho * \pi} - \ve$. As $\rho$ is a probability measure, we have
\bglnoz
  \norm{\pi_x - \rho * \pi}
  &\leq& \ve + \abs{ \int_{\cG^x} f(\eta) d \pi_x(\eta) - \int_{\cG^x} \rukl{ \int_{\cG^{\bfs(\zeta)}} f(\zeta \eta) d \pi_{\bfs(\zeta)}(\eta) } d \rho(\zeta) } \\
   &\leq& \ve + \int_{\cG^x} \abs{ \int_{\cG^x} f(\eta) d \pi_x(\eta) - \int_{\cG^{\bfs(\zeta)}} f(\zeta \eta) d \pi_{\bfs(\zeta)}(\eta) } d \rho(\zeta) \\
  &\leq& \ve + \int_{\cG^x} \norm{ \pi_x - \zeta \pi_{\bfs(\zeta)} } d \rho(\zeta)
\eglnoz
where $\zeta\pi_{\bfs(\zeta)}$ is a measure on $\cG^{\bft(\zeta)} = \cG^x$. This completes the proof.
\end{itemize}
\eproof

\blemma
\label{lemma_conv-meas2}
Let $\cG$ be a Borel semigroupoid with Borel systems $\pi = (\pi_x)_{x \in \cGu}$ and $\rho = (\rho_x)_{x \in \cGu}$ of  probability measures on $\cG$.
Let $\lambda = (\lambda^x)_{x \in \cGu}$ be a quasi-Haar system on $\cG$ in b) and c) below.
\begin{itemize}
\item[\rm a)] $\rho * \pi$ is a Borel system of probability measures on $\cG$.
  \item[\rm b)]  If $\pi_y \prec \lambda^y$ for all $y \in \cGu$, then for every $x \in \cGu$ and every Borel measure $\sigma$ on $\cG^x$, we have $\sigma * \pi \prec \lambda^x$.
  \item[\rm c)] If $\rho_x \prec \lambda^x$ for all $x \in \cGu$, and if both maps $(\gamma,\eta) \in \cG^2 \ma \rukl{ d (\gamma \pi_{\bfs(\gamma)}) / d \lambda^{\bft(\gamma)} }(\eta)   $ and $(\gamma,\eta) \in\cG^2 \ma \rukl{ d (\gamma \rho_{\bfs(\gamma)}) / d \lambda^{\bft(\gamma)} }(\eta)$ are Borel, then so is the map $(\gamma,\eta)\in\cG^2 \ma \rukl{ d (\gamma (\rho * \pi)_{\bfs(\gamma)}) / d \lambda^{\bft(\gamma)} }(\eta)$.
\end{itemize}
\elemma
\bproof\noindent
\begin{itemize}
  \item[ a)] Given a non-negative Borel function $f$ on $\cG$,  the non-negative map $\zeta \in \cG \ma \spkl{\zeta \pi_{\bfs(\zeta)},f}  $ is Borel. Therefore the map $ x\in \cGu \ma \int_{\cG^x} \spkl{\zeta \pi_{\bfs(\zeta)},f} d \rho_x(\zeta) $ is Borel since $\rho$ is a Borel system of measures.
  \item[ b)] This follows from $\spkl{\sigma * \pi,f} = \int_{\cG^x} \rukl{ \int_{\cG^x} f d (\zeta \pi_{\bfs(\zeta)}) } d \sigma(\zeta)$ for each non-negative Borel function $f$ on $\cG^x$, and $\zeta \pi_{\bfs(\zeta)} \prec \zeta \lambda^{\bfs(\zeta)} \prec \lambda^{\bft(\zeta)} = \lambda^x$.
  \item[ c)] Let $u_{\gamma} \defeq d (\gamma \pi_{\bfs(\gamma)}) / d \lambda^{\bft(\gamma)}$ and $v_{\gamma} \defeq d (\gamma \rho_{\bfs(\gamma)}) / d \lambda^{\bft(\gamma)}$. Then for each non-negative Borel function $f$ on $\cG$ and $\gamma \in \cG$ with $\bfs(\gamma) = x$, we have
\bglnoz
  \spkl{\gamma (\rho * \pi)_x,f}
  &=& \int_{\cG^x} \rukl{ \int_{\cG} f(\gamma \zeta \eta) d \pi_{\bfs(\zeta)}(\eta) } d \rho_x(\zeta)
  = \int_{\cG} \rukl{ \int_{\cG} f(\xi \eta) d \pi_{\bfs(\xi)}(\eta) } d (\gamma \rho_x)(\xi) \\
  &=& \int_{\cG} \rukl{ \int_{\cG} f d (\xi \pi_{\bfs(\xi)}) } d (\gamma \rho_x)(\xi)
  = \int_{\cG} \rukl{ \int_{\cG} f u_{\xi} d \lambda^{\bft(\xi)} } v_{\gamma}(\xi) d \lambda^{\bft(\gamma)}(\xi) \\
  &=& \int_{\cG} f(\omega) \rukl{ \int_{\cG} u_{\xi}(\omega) v_{\gamma}(\xi) d \lambda^{\bft(\gamma)}(\xi) }  d \lambda^{\bft(\gamma)}(\omega).
\eglnoz
Hence we have $\rukl{ d (\gamma (\rho * \pi)_{\bfs(\gamma)}) / d \lambda^{\bft(\gamma)} }(\omega) = \int_{\cG} u_{\xi}(\omega) v_{\gamma}(\xi) d \lambda^{\bft(\gamma)}(\xi)$. It follows that the map $(\gamma,\omega)\in \cG^2 \ma \rukl{ d (\gamma (\rho * \pi)_{\bfs(\gamma)}) / d \lambda^{\bft(\gamma)} }(\omega) $ is Borel.

\end{itemize}
\eproof
\blemma
\label{lemma_conv-meas3} Let $\cG$ be a locally compact semigroupoid with two continuous systems $\pi= (\pi_x)_{x \in \cGu}$ and $\rho = (\rho_x)_{x \in \cGu}$  of probability measures on $\cG$. Then
\begin{itemize}
\item[\rm a)] $\rho * \pi$ is a continuous system of probability measures.
  \item[\rm b)]  If $\pi$ and $\rho$ have the property that for every compact subset $K \subseteq \cGu$, the supports $\suppc_K(\pi)$ and $\suppc_K(\rho)$ are compact, then $\rho * \pi$ has this property as well.
\item[\rm c)] Let $\lambda = (\lambda^x)_{x \in \cGu}$ be a quasi-Haar system on $\cG$.  If $\pi_x \prec \lambda^x$ and $\rho_x \prec \lambda^x$ for all $x \in \cGu$, and if both maps $(\gamma,\eta) \in \cG^2 \ma \rukl{ d (\gamma \pi_{\bfs(\gamma)}) / d \lambda^{\bft(\gamma)} }(\eta) \in \Rz$ and $(\gamma,\eta)\in\cG^2 \ma \rukl{ d (\gamma \rho_{\bfs(\gamma)}) / d \lambda^{\bft(\gamma)} }(\eta) \in \Rz$ are bounded and continuous, then the map $(\gamma,\eta)\in \cG^2 \ma \rukl{ d (\gamma (\rho * \pi)_{\bfs(\gamma)}) / d \lambda^{\bft(\gamma)} }(\eta) \in \Rz$ is also bounded and continuous.
\end{itemize}
\elemma

\bproof\noindent\begin{itemize}
  \item[ a)] Since $\cG$ is locally compact and second countable, we know that the map $x\in\cGu \ma \spkl{\pi_x,f}\in \Cz $ is continuous for every $f \in C_b(\cG)$ (not only for $f \in C_c(\cG)$), and likewise for $\rho$. A similar argument as in the proof of \cite[Chapter~III, Lemma~1.1]{para} shows that for every $f \in C_b(\cG)$, the map $\zeta\in \cG \ma \spkl{\zeta \pi_{\bfs(\zeta)},f}\in \Cz$ is continuous (and obviously bounded, too). Hence, as $\rho$ is a continuous system of measures, we conclude that the map $x\in \cGu \ma \spkl{(\rho * \pi)_x,f} = \int_{\cG^x} \spkl{\zeta \pi_{\bfs(\zeta)},f} d \rho_x(\zeta)\in \Cz$  is continuous.
  \item[ b)] Given two subsets $A$ and $B$ of $\cG$, we write $A \cdot B$ for the image of $(A \times B) \cap \cG^{(2)}$ under the composition map $ \: \cG^{(2)} \to \cG$. Obviously, $A \cdot B$ is compact if both $A$ and $B$ are compact. Now given a compact subset $K$ of $\cGu$, it is obvious that $\suppc_K(\rho * \pi) \subseteq \suppc_K(\rho) \cdot \suppc_K(\pi)$.

  \item[ c)] This follows from the same computation as in the proof of Lemma \ref{lemma_conv-meas2} c) and a similar argument to the proof of \cite[Chapter~III, Lemma~1.1]{para}.

\end{itemize}
\eproof

\section{Liouville property implies Reiter's condition}
\label{LR_sgpd}

We are now ready to reveal the relationship between Reiter's condition and the Liouville property for semigroupoids. We do so in the measurable and topological setting.
 In this section, we prove that for semigroupoids equipped with a quasi-Haar system, the Liouville property implies Reiter's condition. This extends two results of Kaimanovich (Theorem~4.2 and Theorem~6.1 in \cite{K2}) for groupoids with a Haar system. We first establish the connection between our setting and the one in \cite[\S~3]{K2}.

Let $(\cG,\lambda,\mu) $ be a measure semigroupoid and $\pi = (\pi_x)_{x\in \cGu}$  a $\lambda$-adapted
Borel system of probability measures on $\cG$. We observe that, for every $x \in \cGu$, the fibrewise Markov operator $P_x: L^\infty(\cG^x,\lambda_x) \rightarrow L^\infty(\cG^x,\lambda_x)$ has a predual operator given by
 $\theta \in L^1(\cG^x,\lambda_x) \ma \theta P_x:=\theta * \pi\in L^1(\cG^x,\lambda_x)$   since
$$\spkl{\theta,P_x(f)} = \int_{\cG^x} \rukl{ \int_{\cG} f(\zeta \eta) d \pi_{\bfs(\zeta)}(\eta) } d \theta(\zeta) = \spkl{\theta * \pi,f} \qquad (f \in L^\infty(\cG^x,\lambda_x))$$
(cf. \cite[\S~3]{K2}). We have the $k$-th iterate $\theta P_x^k = \theta * \pi^k$.

\btheo
\label{L-->R_m}
Let $(\cG,\lambda,\mu)$ be a measure semigroupoid. If $(\cG,\lambda,\mu)$ is Liouville. Then it is Reiter.
\etheo
\bproof
Let $\pi = (\pi_x)_{x\in \cGu}$ be a $\lambda$-adapted Borel system of probability measures such that $H_{\pi,x}(\cG) = \Cz\mathbf 1$ for $\mu$-a.e. $x \in \cGu$. Moreover, let $\theta = (\theta^x)_{x \in \cGu}$ be any $\lambda$-adapted Borel system of probability measures. For instance, we could take $\theta^x = \pi_x$. Set $\theta_n^x \defeq \tfrac{1}{n+1} \sum_{k=0}^n \theta^x * \pi^k$. This is, for every $n \in \Nz$, a $\lambda$-adapted Borel system of probability measures by Lemma~\ref{lemma_conv-meas2}~a), b) and c). We have for every $\gamma \in \cG$:
\bgloz
  \norm{\theta_n^{\bft(\gamma)} - \gamma \theta_n^{\bfs(\gamma)}}
  = \norm{ \tfrac{1}{n+1} \sum_{k=0}^n
  \theta^{\bft(\gamma)} * \pi^{k} - \gamma(\theta^{\bfs(\gamma)} * \pi^{k})
  }
  = \norm{ \tfrac{1}{n+1} \sum_{k=0}^n
  (\theta^{\bft(\gamma)} - \gamma \theta^{\bfs(\gamma)}) * \pi^{k}
  }.
\egloz
By assumption, we have $H_{\pi,x}(\cG) = \Cz \mathbf{1}$ for $\mu$-a.e. $x \in \cGu$. In the language of \cite{K1}, this means that for $\mu$-a.e. $x \in \cGu$, the Poisson boundary of $P_x$ is trivial. Hence \cite[Theorem~2.8]{K1} implies that for $\mu$-a.e. $x \in \cGu$ and $\lambda^x$-a.e. $\gamma \in \cG^x$, we have
\bgloz
  \lim_{n \to \infty}
  \norm{ \tfrac{1}{n+1} \sum_{k=0}^n
  (\theta^{\bft(\gamma)} - \gamma \theta^{\bfs(\gamma)}) * \pi^{k}
  } = \lim_{n \to \infty} \norm{ \tfrac{1}{n+1} \sum_{k=0}^n
  (\theta^{\bft(\gamma)}P_{\bft(\gamma)}^k - \gamma \theta^{\bfs(\gamma)}P_{\bft(\gamma)}^k)
  } = 0.
\egloz
By Remark \ref{8-7}, $(\cG,\lambda,\mu)$ is Reiter.
  \eproof

In the topological setting, we have the following analogous result.

\btheo
\label{L-->R_t}
Let $(\cG,\lambda)$ be a locally compact semigroupoid with quasi-Haar system $\lambda$.
\begin{itemize}
  \item[\rm 1)] If $(\cG,\lambda)$ is Liouville, then it is Reiter.
  \item[\rm 2)] If $(\cG,\lambda)$ is tight Liouville, then it is tight Reiter.
  \item[\rm 3)] If $(\cG,\lambda)$ is continuous Liouville, then it is uniform Reiter.
\end{itemize}
\etheo
\bproof
For 1), let $\pi = (\pi_x)_{x \in \cGu}$ be a $\lambda$-adapted continuous system of probability measures on $\cG$ such that $H_{\pi,x}(\cG) = \Cz\mathbf 1$ for all $x \in \cGu$. Define a new family of measures $\rho = (\rho_x)_{x \in \cGu}$ by setting $\rho_x \defeq \tfrac{1}{2}(\pi_x + (\pi * \pi)_x)$. By Lemma~\ref{lemma_conv-meas2}~b), c) and Lemma \ref{lemma_conv-meas3}~a),  $\rho$ is again a $\lambda$-adapted continuous system of probability measures on $\cG$. Moreover, let $\theta = (\theta^x)_{x \in \cGu}$ be any $\lambda$-adapted continuous system of probability measures. For instance, we could take $\theta^x = \pi_x$. We define, for each $n \in \Nz$, $\theta_n \defeq \theta * \rho^{n}$. By our assumption, $H_{\pi,x}(\cG) = \Cz \mathbf{1}$ for all $x \in \cGu$, or in the language of \cite{K1}, the Poisson boundary of $P_x$ is trivial for all $x \in \cGu$. Hence, combining \cite[Theorem~2.6]{K1} and \cite[Theorem~2.7]{K1} as in \cite[Proof of Theorem~6.1]{K2}, we obtain for every $\gamma \in \cG$:
\bgln
  \lim_{n \to \infty} \norm{\theta_n^{\bft(\gamma)} - \gamma \theta_n^{\bfs(\gamma)}}
  &=& \lim_{n \to \infty} \norm{(\theta^{\bft(\gamma)} - \gamma \theta^{\bfs(\gamma)}) * \rho^n} \nonumber \\
\label{L-->R_t-conv}
  &=& \lim_{n \to \infty} \norm{2^{-n} \sum_{k=0}^n \binom{n}{k} (\theta^{\bft(\gamma)} - \gamma \theta^{\bfs(\gamma)}) * \pi^{n+k}} = 0.
\egln
This proves 1).

For 2), note that if we choose $\theta$ and $\pi$ with the property that $\suppc_K(\theta)$ and $\suppc_K(\pi)$ are compact for all compact subsets $K \subseteq \cGu$, then for every $n \in \Nz$, $\theta_n$ has the same property due to Lemma~\ref{lemma_conv-meas3}~b).

For 3), we choose $\theta$ and $\pi$ such that $\cG^2 \to \Rz, \, (\gamma,\eta) \ma \rukl{ d (\gamma \theta_{\bfs(\gamma)}) / d \lambda^{\bft(\gamma)} }(\eta)$ and $\cG^2 \to \Rz, \, (\gamma,\eta) \ma \rukl{ d (\gamma \pi_{\bfs(\gamma)}) / d \lambda^{\bft(\gamma)} }(\eta)$ are bounded and continuous. Then, by Lemma~\ref{lemma_conv-meas3}~c), we know that for every $n \in \Nz$, $\theta_n$ has the same property. Hence for every $n \in \Nz$, the map $\gamma\in \cG \ma  \norm{\theta_n^{\bft(\gamma)} - \gamma \theta_n^{\bfs(\gamma)}}
\in \Rz$ is continuous. Moreover, these maps form a decreasing sequence (in $n$) by Lemma~\ref{lemma_conv-meas1}~a). Therefore, Dini's theorem implies that the convergence in \eqref{L-->R_t-conv} is uniform on compact subsets of $\cG$.
\eproof

\section{Reiter's condition implies Liouville property}
\label{RL_sgpd}

We are going to prove that Reiter's condition implies the Liouville property for semigroupoids,
both in the measurable and topological setting. In the special
case of groupoids, this proves Kaimanovich's conjecture and the details will be given in the
next section. A crucial construction in the proof is to replace, in Reiter's condition on the semigroupoid $\cG$,
the Borel systems $(\theta_n)_{n\in \mathbb{N}}$ of approximately invariant measures  by a single system $\pi = (\pi_x)_{x\in \cGu}$ so that the convolution powers $(\pi_x^n)_{x \in \cGu}$ play the role of the sequence
$(\theta_n^x)_{x \in \cGu}$. In the proof of Proposition~\ref{mgpd-prop-am-->L}, we follow the strategy adopted in \cite{KV} for the case of discrete groups.  We begin with measure semigroupoids.


\bprop
\label{mgpd-prop-am-->L}
Let $(\cG,\lambda,\mu)$ be a measure semigroupoid. If $(\cG,\lambda,\mu)$ is Reiter, then there exists a $\lambda$-adapted Borel system of probability measures $\pi = (\pi_x)_{x \in \cGu}$ on $\cG$ such that for $\mu$-a.e. $x \in \cGu$ and $\lambda^x$-a.e. $\gamma \in \cG^x$, we have $\lim_{i \to \infty} \norm{\pi^{i}_{\bft(\gamma)} - \gamma \pi^{i}_{\bfs(\gamma)}} = 0$.
\eprop
\bproof
Since $L^1(\cGu,\mu)$ is separable (see Remark~\ref{sep}), there exists a sequence of finite subsets $(\cF_i)_i$ of $L^1(\cGu,\mu)_+$ such that $\cF_1 \subseteq \cF_2 \subseteq \dotso$ and $\bigcup_{i=1}^{\infty} \cF_i$ is dense in $L^1(\cGu,\mu)_+$. Choose sequences $(t_i)_i$ and $(\ve_i)_i$ in $(0,1)$ such that $\sum_{i=1}^{\infty} t_i = 1$, $\ve_1 > \ve_2 > \ve_3 > \dotso$ and $\lim_{i \to \infty} \ve_i = 0$. As $\lambda$ is proper (Definition~\ref{proper}), we can find an increasing sequence of Borel subsets $A_i$ of $\cG$ such that $\cG = \bigcup_{i=1}^{\infty} A_i$ and
\begin{equation}\label{ci}
C_i \defeq \sup_{x \in \cGu} \lambda^x(A_i) \in (0, \infty).
\end{equation}
Choose a sequence $(n_i)_i$ of natural numbers such that $n_1 < n_2 < \dotso$ and $(t_1 + \dotsb + t_{i-1})^{n_i} < \tfrac{\ve_i}{C_i}$.

Let us now choose inductively $\lambda$-adapted Borel systems of probability measures $\theta_i = (\theta_i^x)_{x \in \cGu}$ on $\cG$. Let $\theta_1$ be any $\lambda$-adapted Borel system of probability measures on $\cG$. Now assume that $\theta_1, \dotsc, \theta_{m-1}$ have been chosen. Let
$$
\Theta_{m-1} \defeq \menge{\theta_{k_1} * \dotsb * \theta_{k_j}}{1 \leq j \leq n_m - 1, \, k_j \in \gekl{1, \dotsc, m-1}}.
$$
Every $\rho \in \Theta_{m-1}$ is $\lambda$-adapted by Lemma~\ref{lemma_conv-meas2}~c). Let $u^x \defeq d \rho^x / d \lambda^x$.
Then   $\int_{\cG^x} u^x(\zeta) d \lambda^x(\zeta) =\int_{\cG^x} d \rho^x(\zeta) =1$ as $\rho^x$ is a probability measure. For $f \in \cF_{m-1}$ and $\rho \in \Theta_{m-1}$, let $\varphi_{f,\rho}(\zeta) = u^{\bft(\zeta)} f(\bft(\zeta))$.
Then $\varphi_{f,\rho}$ is Borel because $\rho$ is $\lambda$-adapted, and we have $\varphi_{f,\rho} \in L^1(\cG, \mu\star\lambda)$ since
\bgloz
  \int_{\cGu} \rukl{ \int_{\cG^x} \varphi_{f,\rho}(\zeta) d \lambda^x(\zeta) } d \mu(x)
  = \int_{\cGu} \rukl{ \int_{\cG^x} u^x(\zeta) d \lambda^x(\zeta) } f(x) d \mu(x)
  =  \int_{\cGu} f(x) d \mu(x) < \infty.
\egloz

For each $\rho \in \Theta_{m-1}$, let $v_{\gamma}(\eta) \defeq \rukl{d (\gamma \rho^{\bfs(\gamma)}) / d \lambda^{\bft(\gamma)}} (\eta)$.  Since $\gamma \rho^{\bfs(\gamma)}$ is a probability measure, we have
\begin{equation}\label{v}
\int_{\cG^x} v_{\gamma}(\eta) d \lambda^x(\eta) =  \int_{\cG^x} d (\gamma \rho^{\bfs(\gamma)}) =1.
\end{equation}
For $f \in \cF_{m-1}$ and $\rho \in \Theta_{m-1}$, set
$
  \phi_{f,\rho}(\eta) = \int_{\cG} v_{\gamma}(\eta) 1_{A_m}(\gamma) f(\bft(\eta)) d \lambda^{\bft(\eta)}(\gamma)
$. Again, $\phi_{f,\rho}$ is Borel because $\rho$ is $\lambda$-adapted, and $\phi_{f,\rho}$ lies in $L^1(\cG, \mu\star\lambda)$
as (\ref{ci}) and (\ref{v}) imply
\bglnoz
  && \int_{\cGu} \rukl{ \int_{\cG^x} \phi_{f,\rho}(\eta) d \lambda^x(\eta) } d \mu(x) \\
  &=& \int_{\cGu} \rukl{ \int_{\cG^x} \rukl{ \int_{\cG^{\bft(\eta)}} v_{\gamma}(\eta) 1_{A_m}(\gamma) f(\bft(\eta)) d \lambda^{\bft(\eta)}(\gamma) } d \lambda^x(\eta) } d \mu(x) \\
  &=& \int_{\cGu} \rukl{ \int_{\cG^x} \rukl{ \int_{\cG^x} v_{\gamma}(\eta) d \lambda^x(\eta) } 1_{A_m}(\gamma)
  d \lambda^x(\gamma) } f(x) d \mu(x)\\
  &\leq& C_m \cdot \int_{\cGu} f(x) d \mu(x) < \infty.
\eglnoz

Since $(\cG, \lambda,\mu)$ is Reiter,  we can choose $\theta_m$ such that
\bgl
\label{tm1}
  \int_{\cG} \norm{\theta_m^{\bft(\gamma)} - \gamma \theta_m^{\bfs(\gamma)}} 1_{A_m}(\gamma) f(\bft(\gamma)) d (\mu\star\lambda)(\gamma) < \ve_m \ {\rm for \ all} \ f \in \cF_{m-1},
\egl
\bgl
\label{tm2}
  \int_{\cG} \norm{\theta_m^{\bft(\zeta)} - \zeta \theta_m^{\bfs(\zeta)}} \varphi_{f,\rho}(\zeta) d (\mu\star\lambda)(\zeta) < \tfrac{\ve_m}{C_m} \ {\rm for \ all} \ f \in \cF_{m-1} \ {\rm and} \ \rho \in \Theta_{m-1}
\egl
and
\bgl
\label{tm3}
  \int_{\cG} \norm{\theta_m^{\bft(\eta)} - \eta \theta_m^{\bfs(\eta)}} \phi_{f,\rho}(\eta) d (\mu\star\lambda)(\eta) < \ve_m \ {\rm for \ all} \ f \in \cF_{m-1} \ {\rm and} \ \rho \in \Theta_{m-1}.
\egl

Now set $\pi_x \defeq \sum_{i=1}^{\infty} t_i \theta_i^x$ for all $x \in \cGu$. $\pi = (\pi_x)$ is again a $\lambda$-adapted Borel system of probability measures.

Take $m \in \Nz$ and $f \in \cF_{m-1}$. Set $n \defeq n_m$. Write the $n$-fold convolution $\pi^{n}_x = (\pi_1)_x + (\pi_2)_x$ where
\bgln
\label{pi1x}
  (\pi_1)_x &=& \sum_{\substack{\bk \in \Nz^n \\ \max(\bk) < m}}
  t_{k_1} \cdots t_{k_n} (\theta_{k_1} * \dotsb * \theta_{k_n})^x, \\
\label{pi2x}
  (\pi_2)_x &=& \pi^{n}_x - (\pi_1)_x = \sum_{\substack{\bk \in \Nz^n \\ \max(\bk) \geq m}}
  t_{k_1} \cdots t_{k_n} (\theta_{k_1} * \dotsb * \theta_{k_n})^x.
\egln
Then $\norm{(\pi_1)_x} \leq (t_1 + \dotsb + t_{m-1})^{n_m} < \tfrac{\ve_m}{C_m}$ for all $x \in \cGu$ and hence
\bgln
\label{mgpd-pi1}
  && \int_{\cGu} \rukl{ \int_{\cG^x} \norm{(\pi_1)_{\bft(\gamma)} - \gamma (\pi_1)_{\bfs(\gamma)}} 1_{A_m}(\gamma) d \lambda^x(\gamma) } f(x) d \mu(x) \\
  &<& \int_{\cGu} \rukl{ \int_{\cG^x} 2 \tfrac{\ve_m}{C_m} 1_{A_m}(\gamma) d \lambda^x(\gamma) } f(x) d \mu(x) \leq 2 \ve_m. \nonumber
\egln
For $\bk \in \Nz^n$ with $\max(\bk) \geq m$, $\theta_{k_1} * \dotsb * \theta_{k_n}$ is of the form $\theta * \tau$ for some $\theta = \theta_k$ with $k \geq m$, or it is of the form $\rho * \theta * \tau$ for some $\rho \in \Theta_{m-1}$ and $\theta = \theta_k$ with $k \geq m$.
By \eqref{tm1}, we have
\bgl
\label{mgpd-ineq1}
  \int_{\cGu} \rukl{ \int_{\cG^x} \norm{\theta^{\bft(\gamma)} - \gamma \theta^{\bfs(\gamma)}} 1_{A_m}(\gamma) d \lambda^x(\gamma) } f(x) d \mu(x) < \ve_m.
\egl
For every $x \in \cGu$, Lemma~\ref{lemma_conv-meas1}~b) implies
\bgl
\label{mgpd-ineq211}
  \norm{\theta^x - (\rho * \theta)^x}
  \leq \int_{\cG^x} \norm{\theta^{\bft(\zeta)} - \zeta \theta^{\bfs (\zeta)}} d \rho^x(\zeta).
\egl
Therefore we have
\bgln
\label{mgpd-ineq21}
  && \int_{\cGu} \rukl{ \int_{\cG^x} \norm{\theta^x - (\rho * \theta)^x} 1_{A_m}(\gamma) d \lambda^x(\gamma) } f(x) d \mu(x) \\
  &\leq& C_m \int_{\cGu} \norm{\theta^x - (\rho * \theta)^x} f(x) d \mu(x) \nonumber \\
  &\overset{\eqref{mgpd-ineq211}}{\leq}& C_m \int_{\cGu} \rukl{ \int_{\cG^x} \norm{\theta^{\bft(\zeta)} - \zeta \theta^{\bfs (\zeta)}} d \rho^x(\zeta) } f(x) d \mu(x) \nonumber \\
  &=& C_m \int_{\cGu} \rukl{ \int_{\cG^x} \norm{\theta^{\bft(\zeta)} - \zeta \theta^{\bfs (\zeta)}} u^x(\zeta) f(x) d \lambda^x(\zeta)} d \mu(x) \nonumber \\
  &=& C_m \int_{\cG} \norm{\theta^{\bft(\zeta)} - \zeta \theta^{\bfs (\zeta)}} \varphi_{f,\rho}(\zeta) d (\mu\star\lambda)(\zeta) \overset{\eqref{tm2}}{<} \ve_m. \nonumber
\egln
For every $\gamma \in \cG$, Lemma~\ref{lemma_conv-meas1}~b) implies that
\bgl
\label{mgpd-ineq221}
  \norm{\theta^{\bft(\gamma)} - \gamma (\rho * \theta)^{\bfs(\gamma)}}
  = \norm{\theta^{\bft(\gamma)} - (\gamma \rho^{\bfs(\gamma)}) * \theta}
  \leq \int_{\cG^{\bft(\gamma)}} \norm{\theta^{\bft(\gamma)} - \eta \theta^{\bfs(\eta)}} d (\gamma \rho^{\bfs(\gamma)})(\eta).
\egl
Using $v_{\gamma}(\eta) \defeq \rukl{d (\gamma \rho^{\bfs(\gamma)}) / d \lambda^{\bft(\gamma)}} (\eta)$
and $
  \phi_{f,\rho}(\eta) = \int_{\cG} v_{\gamma}(\eta) 1_{A_m}(\gamma) f(\bft(\eta)) d \lambda^{\bft(\eta)}(\gamma)
$,
a computation analogous to (\ref{mgpd-ineq21})
gives
\bgln
\label{mgpd-ineq22}
  && \int_{\cGu} \rukl{ \int_{\cG^x} \norm{\theta^{\bft(\gamma)} - \gamma (\rho * \theta)^{\bfs(\gamma)}} 1_{A_m}(\gamma) d \lambda^x(\gamma) } f(x) d \mu(x) \\
  &\overset{\eqref{mgpd-ineq221}}{\leq}&
  \int_{\cGu} \rukl{ \int_{\cG^x} \rukl{ \int_{\cG^x} \norm{\theta^{\bft(\gamma)} - \eta \theta^{\bfs(\eta)}} d (\gamma \rho^{\bfs(\gamma)})(\eta) } 1_{A_m}(\gamma) d \lambda^x(\gamma) } f(x) d \mu(x)
   \nonumber \\
  &=& \int_{\cG} \norm{\theta^{\bft(\eta)} - \eta \theta^{\bfs(\eta)}} \phi_{f,\rho}(\eta) d (\mu\star\lambda)(\eta) \overset{\eqref{tm3}}{<} \ve_m. \nonumber
\egln
Combining \eqref{mgpd-ineq21} and \eqref{mgpd-ineq22}, we get
  $$\int_{\cGu} \rukl{ \int_{\cG^x} \norm{(\rho * \theta)^{\bft(\gamma)} - \gamma (\rho * \theta)^{\bfs(\gamma)}} 1_{A_m}(\gamma) d \lambda^x(\gamma) } f(x) d \mu(x) < 2 \ve_m$$
and also for all $m'\leq m$,
\begin{equation}
\label{mgpd-ineq2}
 \int_{\cGu} \rukl{ \int_{\cG^x} \norm{(\rho * \theta)^{\bft(\gamma)} - \gamma (\rho * \theta)^{\bfs(\gamma)}} 1_{A_{m'}}(\gamma) d \lambda^x(\gamma) } f(x) d \mu(x) < 2 \ve_m.
\end{equation}
With the help of Lemma~\ref{lemma_conv-meas1}~a), we deduce from \eqref{mgpd-ineq1} and \eqref{mgpd-ineq2} that
\bgloz
  \int_{\cGu} \rukl{ \int_{\cG^x} \norm{(\theta * \tau)^{\bft(\gamma)} - \gamma (\theta * \tau)^{\bfs(\gamma)}} 1_{A_{m'}}(\gamma) d \lambda^x(\gamma) } f(x) d \mu(x) < \ve_m
\egloz
and
\bgloz
  \int_{\cGu} \rukl{ \int_{\cG^x} \norm{(\rho * \theta * \tau)^{\bft(\gamma)} - \gamma (\rho * \theta * \tau)^{\bfs(\gamma)}} 1_{A_{m'}}(\gamma) d \lambda^x(\gamma) } f(x) d \mu(x) < 2 \ve_m.
\egloz
As every summand in $\pi_2$ is of the form $\theta * \tau$ or $\rho * \theta * \tau$, we obtain
\bgln
\label{mgpd-pi2}
  && \int_{\cGu} \rukl{ \int_{\cG^x} \norm{(\pi_2)_{\bft(\gamma)} - \gamma  (\pi_2)_{\bfs(\gamma)}} 1_{A_{m'}}(\gamma) d \lambda^x(\gamma) } f(x) d \mu(x) \\
  &<& \sum_{\substack{\bk \in \Nz^n \\ \max(\bk) \geq m}}
  t_{k_1} \cdots t_{k_n} \cdot 2 \ve_m
  \leq \rukl{\sum_{i=1}^{\infty} t_i}^n \cdot 2 \ve_m = 2 \ve_m. \nonumber
\egln
Finally, combining \eqref{mgpd-pi1} and \eqref{mgpd-pi2}, using Lemma~\ref{lemma_conv-meas1}~a), we obtain, for
$m'\leq m$,
\bgloz
  \int_{\cGu} \rukl{ \int_{\cG^x} \norm{\pi^{i}_{\bft(\gamma)} - \gamma  \pi^{i}_{\bfs(\gamma)}} 1_{A_{m'}}(\gamma) d \lambda^x(\gamma) } f(x) d \mu(x)
  < 4 \ve_m \ {\rm for \ all} \ i \geq n=n_m.
\egloz
Since $\lim_{m\rightarrow \infty} \ve_m =0$, we have, for every $m' \in \Nz$ and $f \in \cF_{m'-1}$,
\bgl
\label{mgpd-->0}
  \lim_{i \to \infty} \int_{\cGu} \rukl{ \int_{\cG^x} \norm{\pi^{i}_{\bft(\gamma)} - \gamma  \pi^{i}_{\bfs(\gamma)}} 1_{A_{m'}}(\gamma) d \lambda^x(\gamma) } f(x) d \mu(x) = 0
\egl
which gives
$$ \lim_{i \to \infty} \int_{\cGu} \rukl{ \int_{\cG^x} \norm{\pi^{i}_{\bft(\gamma)} - \gamma  \pi^{i}_{\bfs(\gamma)}}  d \lambda^x(\gamma) } f(x) d \mu(x) = 0$$
for all $f \in \bigcup_i \cF_i$, and hence for all $f \in L^1(\cGu,\mu)$. It follows that for $\mu$-a.e. $x \in \cGu$ and $\lambda^x$-a.e. $\gamma \in \cG^x$, we have $\lim_{i \to \infty} \norm{\pi^{i}_{\bft(\gamma)} - \gamma \pi^{i}_{\bfs(\gamma)}}= 0$. Here we have used the fact that the sequence $\rukl{\norm{\pi^{i}_{\bft(\gamma)} - \gamma \pi^{i}_{\bfs(\gamma)}}}_i$ is decreasing by Lemma~\ref{lemma_conv-meas1}~a).
\eproof

\blemma
\label{mgpd-lem-am-->L}
Let $(\cG,\lambda,\mu)$ be a measure semigroupoid. If there exists a $\lambda$-adapted Borel system of probability measures $\pi = (\pi_x)_{x \in \cGu}$  such that $\lim_{i \to \infty} \norm{\pi^{i}_{\bft(\gamma)} - \gamma \pi^{i}_{\bfs(\gamma)}} = 0$
for $\lambda^x$-a.e. $\gamma \in \cG^x$ and $\mu$-a.e. $x \in \cGu$,
then we have  $H_{\pi,x}(\cG) = \Cz \mathbf{1}$ for $\mu$-a.e. $x \in \cGu$.
\elemma
\bproof
Take $x \in \cGu$ such that $\lim_{i \to \infty} \norm{\pi^{i}_{\bft(\gamma)} - \gamma \pi^{i}_{\bfs(\gamma)}} = 0$ for $\lambda^x$-a.e. $\gamma \in \cG^x$.  Let $f \in L^{\infty}(\cG^x,\lambda^x)$ satisfy $P_x(f) = f$. Then $f = P_x^i(f)$ and, for $\lambda^x$-a.e. $\gamma \in \cG^x$, we have $f(\gamma) = \int_{\cG^x} f(\gamma \zeta) d \pi^{i}_{\bfs(\gamma)} (\zeta)$. Therefore
\bgloz
  \abs{ f(\gamma)-\int_{\cG^x} f(\zeta) d \pi^{i}_x (\zeta)} = \abs{ \spkl{  \gamma \pi^{i}_{\bfs(\gamma)}- \pi^{i}_{\bft(\gamma)}, \,f } } \leq \norm{f}_{\infty} \cdot \norm{\gamma \pi^{i}_{\bfs(\gamma)}- \pi^{i}_{\bft(\gamma)}} \underset{i \to \infty}{\lori} 0
\egloz
 for $\lambda^x$-a.e. $\gamma \in \cG^x$. Hence  $f$ is constant $\lambda^x$-a.e.
\eproof
\rm
From Proposition~\ref{mgpd-prop-am-->L} and Lemma~\ref{mgpd-lem-am-->L}, we have established the following theorem.

\btheo
\label{R-->L_m}
Let $(\cG,\lambda,\mu)$ be a measure semigroupoid. If $(\cG,\lambda,\mu)$ is Reiter, then it is Liouville.
\etheo

We now discuss topological semigroupoids. In addition to the previous approach to the measure semigroupoid case, we need to make use of the uniform and tight Reiter's condition to deduce the existence of one single continuous system $\pi = (\pi_x)_{x \in \cGu}$ in the following key proposition.
\bprop
\label{tgpd-prop-am-->L}
Let $(\cG,\lambda)$ be a locally compact semigroupoid with quasi-Haar system. If $(\cG,\lambda)$ is uniform and tight Reiter, then there exists a $\lambda$-adapted continuous system of probability measures $\pi = (\pi_x)_{x \in \cGu}$ such that
\bgloz
  \lim_{i \to \infty} \norm{\pi^i_{\bft(\gamma)} - \gamma \pi^i_{\bfs(\gamma)}} = 0 \ {\rm for \ all} \ \gamma \in \cG.
\egloz
\eprop
\bproof
Since $\cG$ is $\sigma$-compact (see \S~\ref{sgpd}), there exists a sequence of compact subsets $(C_i)_i$ of $\cG$ such that $C_0 \subseteq C_1 \subseteq C_2 \subseteq \dotso$ and $\cG = \bigcup_{i=0}^{\infty} C_i$. Choose sequences $(t_i)_i$ and $(\ve_i)_i$ in $(0,1)$ such that $\sum_{i=1}^{\infty} t_i = 1$, $\ve_1 > \ve_2 > \ve_3 > \dotso$ and $\lim_{i \to \infty} \ve_i = 0$. Furthermore, choose a sequence $(n_i)_i$ of natural numbers such that $n_1 < n_2 < \dotso$ and $(t_1 + \dotsb + t_{i-1})^{n_i} < \ve_i$.

Let us now inductively choose $\lambda$-adapted continuous systems of probability measures $\theta_i = (\theta_i^x)_{x \in \cGu}$. Let $\theta_1$ be such a continuous system of probability measures with $\norm{\theta_1^{\bft(\gamma)} - \gamma \theta_1^{\bfs(\gamma)}} < \ve_1$ for all $\gamma \in C_0$. Now assume that $\theta_1, \dotsc, \theta_{m-1}$ have been chosen. Let
$$
\Theta_{m-1} \defeq \menge{\theta_{k_1} * \dotsb * \theta_{k_j}}{1 \leq j \leq n_m - 1, \, k_j \in \gekl{1, \dotsc, m-1}}.
$$
Lemma~\ref{lemma_conv-meas2}~c) implies that every $\rho \in \Theta_{m-1}$ is $\lambda$-adapted. Moreover, for $K \subseteq \cGu$ compact, set $\suppc_K(\Theta_{m-1}) \defeq \bigcup_{\rho \in \Theta_{m-1}} \suppc_K(\rho)$. Then, by Lemma~\ref{lemma_conv-meas3}~b), $\suppc_K(\Theta_{m-1})$ is compact for every compact subset $K \subseteq \cGu$. Now choose $\theta_m$ such that
\bgl
\label{tgpd-tm}
 \norm{\theta_m^{\bft(\gamma)} - \gamma \theta_m^{\bfs(\gamma)}} < \ve_m \ {\rm for \ all} \ \gamma \in \suppc_{\bft(C_{m-1})}(\Theta_{m-1}) \cup (C_{m-1} \cdot \suppc_{\bfs(C_{m-1})}(\Theta_{m-1})) \cup C_{m-1}.
\egl
Such $\theta_m$ exist since $\cG$ is uniform and tight Reiter.

Now set $\pi_x \defeq \sum_{i=1}^{\infty} t_i \theta_i^x$ for all $x \in \cGu$. $\pi = (\pi_x)_{x \in \cGu}$ is again a $\lambda$-adapted continuous system of probability measures. Take $m \in \Nz$ and $\gamma \in C_{m-1}$. Set $n \defeq n_m$. Write $\pi^{n}_x = (\pi_1)_x + (\pi_2)_x$ where $(\pi_1)_x$ and $(\pi_2)_x$ are defined as in \eqref{pi1x} and \eqref{pi2x}. Then $\norm{(\pi_1)_x} \leq (t_1 + \dotsb + t_{m-1})^{n_m} < \ve_m$ for all $x \in \cGu$. Therefore, $\norm{\gamma (\pi_1)_{\bfs(\gamma)}} < \ve_m$. Thus we obtain
\bgl
\label{tgpd-pi1}
  \norm{(\pi_1)_{\bft(\gamma)} - \gamma (\pi_1)_{\bfs(\gamma)}} < 2 \ve_m.
\egl
For $\bk \in \Nz^n$ with $\max(\bk) \geq m$, $\theta_{k_1} * \dotsb * \theta_{k_n}$ is of the form $\theta * \tau$ for some $\theta = \theta_k$ with $k \geq m$, or it is of the form $\rho * \theta * \tau$ for some $\rho \in \Theta_{m-1}$ and $\theta = \theta_k$ with $k \geq m$.
By choice of $\theta_i$, we have
\bgl
\label{tgpd-ineq1}
  \norm{\theta^{\bft(\gamma)} - \gamma \theta^{\bfs(\gamma)}} \overset{\eqref{tgpd-tm}}{<} \ve_m.
\egl
Moreover, it follows that from Lemma~\ref{lemma_conv-meas1}~b) that
\bgl
\label{tgpd-ineq21}
  \norm{\theta^{\bft(\gamma)} - (\rho * \theta)^{\bft(\gamma)}} \leq \int_{\supp(\rho^{\bft(\gamma)})} \norm{\theta^{\bft(\zeta)} - \zeta \theta^{\bfs (\zeta)}} d \rho^{\bft(\gamma)}(\zeta) \overset{\eqref{tgpd-tm}}{<} \ve_m.
\egl
Similarly, Lemma~\ref{lemma_conv-meas1}~b) implies that
\bgln
  \norm{\theta^{\bft(\gamma)} - \gamma (\rho * \theta)^{\bfs(\gamma)}}
  &=& \norm{\theta^{\bft(\gamma)} - (\gamma \rho^{\bfs(\gamma)}) * \theta}
  \leq \int_{\cG} \norm{\theta^{\bft(\gamma)} - \zeta \theta^{\bfs (\zeta)}} d (\gamma \rho^{\bfs(\gamma)})(\zeta) \nonumber \\
\label{tgpd-ineq22}
  &=& \int_{\supp(\rho^{\bfs(\gamma)})} \norm{\theta^{\bft(\gamma \eta)} - \gamma \eta \theta^{\bfs (\gamma \eta)}} d \rho^{\bfs(\gamma)}(\eta) \overset{\eqref{tgpd-tm}}{<} \ve_m.
\egln
Combining \eqref{tgpd-ineq21} and \eqref{tgpd-ineq22}, we get
\bgl
\label{tgpd-ineq2}
  \norm{(\rho * \theta)^{\bft(\gamma)} - \gamma (\rho * \theta)^{\bfs(\gamma)}} < 2 \ve_m.
\egl
Therefore, because of \eqref{tgpd-ineq1} and \eqref{tgpd-ineq2}, Lemma~\ref{lemma_conv-meas1}~a) implies $\norm{(\theta * \tau)^{\bft(\gamma)} - \gamma (\theta * \tau)^{\bfs(\gamma)}} < \ve_m$ and $\norm{(\rho * \theta * \tau)^{\bft(\gamma)} - \gamma (\rho * \theta * \tau)^{\bfs(\gamma)}} < 2 \ve_m$. As every summand in $\pi_2$ is of the form $\theta * \tau$ or $\rho * \theta * \tau$, we obtain
\bgl
\label{tgpd-pi2}
  \norm{(\pi_2)_{\bft(\gamma)} - \gamma  (\pi_2)_{\bfs(\gamma)}}
  < \sum_{\substack{\bk \in \Nz^n \\ \max(\bk) \geq m}}
  t_{k_1} \cdots t_{k_n} \cdot 2 \ve_m
  \leq \rukl{\sum_{i=1}^{\infty} t_i}^n \cdot 2 \ve_m = 2 \ve_m.
\egl
Finally, combining \eqref{tgpd-pi1} and \eqref{tgpd-pi2}, we get using Lemma~\ref{lemma_conv-meas1}~a)
\bgloz
  \norm{\pi^i_{\bft(\gamma)} - \gamma  \pi^i_{\bfs(\gamma)}} < 4 \ve_m \ {\rm for \ all} \ i \geq n.
\egloz
On the whole, we deduce that for every $\gamma \in \cG$,
\bgloz
  \lim_{i \to \infty} \norm{\pi^i_{\bft(\gamma)} - \gamma  \pi^i_{\bfs(\gamma)}} = 0.
\egloz
\eproof

\blemma
\label{tgpd-lem-am-->L}
Let $\cG$ be a topological semigroupoid with a quasi-Haar system $\lambda$. Suppose that there exists a $\lambda$-adapted continuous system of probability measures $\pi = (\pi_x)_{x \in \cGu}$  such that $\lim_{i \to \infty} \norm{\pi^i_{\bft(\gamma)} - \gamma \pi^i_{\bfs(\gamma)}} = 0$ for all $\gamma \in \cG^x$.
Then we have  $H_{\pi,x}(\cG) = \Cz \mathbf{1}$ for all $x \in \cGu$.
\elemma
\bproof
Take $x \in \cGu$ and $f \in L^{\infty}(\cG^x,\lambda^x)$ with $P_x(f) = f$. Then $f = P_x^i(f)$ and $f(x) = \int_{\cG^x} f(\zeta) d \pi^i_x (\zeta)$ and also, $f(\gamma) = \int_{\cG^x} f(\gamma \zeta) d \pi^i_{\bfs(\gamma)} (\zeta)$ for all $\gamma \in \cG^x$.  Therefore $\abs{f(x) - f(\gamma)} = \abs{ \spkl{ \pi^i_{\bft(\gamma)} - \gamma \pi^i_{\bfs(\gamma)},f } } \leq \norm{f}_{\infty} \cdot \norm{ \pi^i_{\bft(\gamma)} - \gamma \pi^i_{\bfs(\gamma)} } \underset{i \to \infty}{\lori} 0$. This means that for all $\gamma \in \cG^x$, we have $f(\gamma) = f(x)$ and hence $f$ is constant. This proves $H_{\pi,x}(\cG) = \Cz \mathbf{1}$.
\eproof

Finally, Proposition~\ref{tgpd-prop-am-->L} and Lemma~\ref{tgpd-lem-am-->L} together yield the following
result.
\btheo
\label{R-->L_t}
Let $(\cG,\lambda)$ be a locally compact semigroupoid with quasi-Haar system. If $(\cG,\lambda)$ is  uniform and tight Reiter, then it is Liouville.
\etheo

\section{Special cases}
\label{special}

In this section, we discuss three important special cases, namely groupoids, transformation semigroups and semigroups. We begin with groupoids.

\subsection{Groupoids}

An important consequence of our results is the equivalence  of amenability and the Liouville property for groupoids, in both measurable and topological settings. In particular, we settle a conjecture of Kaimanovich \cite[Conjecture~4.6]{K2} by showing that  amenable groupoids  admit the Liouville property. To be precise, combining Theorem~\ref{L-->R_m} and Theorem~\ref{R-->L_m}, we obtain the following result for measured groupoids.

\btheo
\label{mgpd_LR}
Let $(\cG,\lambda,\mu)$ be a measured groupoid, where $\lambda$ is a left Haar system. Then $(\cG,\lambda,\mu)$ is Liouville if and only if it is Reiter.
\etheo
Since amenability and Reiter's condition are equivalent for groupoids, this theorem establishes the equivalence
of the Liouville property and amenability for measured groupoids. This result has also been proved by B\"{u}hler and Kaimanovich in an unpublished note. However, a similar result for topological groupoids is more subtle, which requires a deeper analysis of
Reiter's condition, given below.

\blemma
\label{gpd_RRR}
Let $(\cG,\lambda)$ be a locally compact groupoid, where $\lambda$ is a left Haar system. The following conditions are equivalent:
\begin{itemize}
  \item[\rm (i)] $(\cG,\lambda)$ is Reiter,
  \item[\rm (ii)] $(\cG,\lambda)$ is uniform Reiter,
  \item[\rm (iii)] $(\cG,\lambda)$ is tight Reiter.
\end{itemize}
\elemma
\bproof
(i) $\Rarr$ (ii) follows from \cite[Theorem~2.14]{Rnew}. (iii) $\Rarr$ (i) is obvious. It remains to prove (ii) $\Rarr$ (iii).

We show that for every $\ve > 0$ and $C \subseteq \cG$ compact, there exists a $\lambda$-adapted continuous system of probability measures $\theta = (\theta^x)_{x \in \cGu}$ on $\cG$ such that $ \norm{\theta^{\bft(\gamma)} - \gamma \theta^{\bfs(\gamma)}} < \ve \ {\rm for \ all} \ \gamma \in C$, and in addition, $\suppc_K(\theta) \ {\rm is \ compact \ for \ every \ compact \ subset} \ K \subseteq \cG$.

Given $\ve > 0$ and $C \subseteq \cG$ compact, since $\cG^{(0)}$ is locally compact and $\sigma$-compact, there exist $K_n \subseteq \cG^{(0)}$ compact with $\mathbf{s}(C) \cup \mathbf{t}(C) \subseteq K_1$, $K_n \subseteq \accentset{\circ}{K_{n+1}}$ for all $n \in \Nz$ and $\cG^{(0)} = \bigcup_{n=1}^{\infty} K_n$.

 Since $(\cG,\lambda)$ is Reiter,  $\cG$ is amenable and hence there exists a {\it topological approximate invariant density} $(g_j)_{j=1}^\infty$ in $C_c(\cG)$, as defined in \cite[Definition~2.7]{Rnew}). Normalizing $g_j \cdot \lambda^x$ for some $j$ and for each $x \in K_2$, we obtain
a $\lambda$-adapted continuous system of measures $\vartheta_2 = (\vartheta_2^x)_{x \in \cGu}$ such that
\bgl
\label{gpd_t2-ineq}
  \norm{\vartheta_2^{\bft(\gamma)} - \gamma \vartheta_2^{\bfs(\gamma)}} < \ve \ {\rm for \ all} \ \gamma \in C
\egl
with the additional properties that $\suppc_{K_2}(\vartheta_2)$ is compact and $\vartheta_2^x$ is a probability measure for all $x \in K_2$. Likewise, for all $n \geq 3$, one can find  $\lambda$-adapted continuous systems of measures $\vartheta_n = (\vartheta_n^x)_{x \in \cGu}$ such that $\suppc_{K_n}(\vartheta_n)$ is compact and $\vartheta_n^x$ is a probability measure for all $x \in K_n$.

Now define $U_2 \defeq \accentset{\circ}{K_2}$ and $U_n \defeq \accentset{\circ}{K_n} \setminus K_{n-2}$ for all $n \geq 3$. By construction, $\cG^{(0)} = \bigcup_{n=2}^{\infty} U_n$. Since $\cG^{(0)}$ is locally compact and $\sigma$-compact, $\cG^{(0)}$ is paracompact and one can find a partition of unity $\{h_n\}_{n\geq 2}$ subordinate to $\{U_n\}_{n\geq 2}$. Define $\theta_n^x \defeq h_n(x) \vartheta_2^x$ for $x \in U_n$ and $\theta_n^x \defeq 0$ for $x \notin U_n$. By construction, these $\theta_n$ are $\lambda$-adapted continuous systems of measures such that $\supp(\theta_n) \defeq \suppc_{\cG^{(0)}}(\theta_n) \subseteq \suppc_{K_n}(\vartheta_n)$. In particular, $\supp(\theta_n)$ is compact.

Set $\theta^x \defeq \sum_{n=2}^{\infty} \theta_n^x$ which is a finite sum since the cover $\gekl{U_n}$ is locally finite. As each $\theta_n$ is $\lambda$-adapted, so is the continuous system $\theta= (\theta^x)_{x \in \cG^{(0)}}$ where $\norm{\theta^x} = \sum_n \norm{\theta_n^x}=\sum_n h_n(x) = 1$.

We claim that $\theta$ has the desired properties.
Indeed, given $\gamma \in C$, both $\bft(\gamma)$ and $\bfs(\gamma)$ lie in $K_1$, hence $h_n(\bft(\gamma)) = h_n(\bfs(\gamma)) = 0$ for all $n \geq 3$, and so $h_2(\bft(\gamma)) = h_2(\bfs(\gamma)) = 1$. Thus $\theta^{\bft(\gamma)} = \theta_2^{\bft(\gamma)} = \vartheta_2^{\bft(\gamma)}$, and also $\theta^{\bfs(\gamma)} = \theta_2^{\bfs(\gamma)} = \vartheta_2^{\bfs(\gamma)}$. It follows that
\bgloz
  \norm{\theta^{\bft(\gamma)} - \gamma \theta^{\bfs(\gamma)}} = \norm{\vartheta_2^{\bft(\gamma)} - \gamma \vartheta_2^{\bfs(\gamma)}} \overset{\eqref{gpd_t2-ineq}}{<} \ve.
\egloz
Finally, given $K \subseteq \cG^{(0)}$ compact, there exists $N \in \Nz$ with $K \subseteq \bigcup_{n=2}^N U_n$, so that $h_n \vert_K \equiv 0$ for all $n \geq N+2$. Hence $\suppc_K(\theta) \subseteq \bigcup_{n=2}^{N+1} \supp(\theta_n)$ which is compact.
\eproof

Combining Theorem~\ref{L-->R_t} and Theorem~\ref{R-->L_t} with the previous lemma, we obtain the equivalence of the
Liouville property and Reiter's condition in the topological setting. This proves the topological analogue of \cite[Conjecture~4.6]{K2}.
\btheo
\label{tgpd_LR}
Let $(\cG,\lambda)$ be a locally compact groupoid, where $\lambda$ is a left Haar system. Then $(\cG,\lambda)$ is Liouville if and only if it is Reiter.
\etheo
\bremark
By convention, all our locally compact groupoids are locally compact and second countable. However, the same proofs as above show that Lemma~\ref{gpd_RRR} and Theorem~\ref{tgpd_LR} also hold for topological groupoids which are locally compact and $\sigma$-compact.
\eremark

\subsection{Transformation semigroups}

We now show the equivalence of the Liouville property and Reiter's condition for transformation semigroups. First,
combining Theorem~\ref{L-->R_m} and Theorem~\ref{R-->L_m}, one obtains the following result in the measure setting.

\btheo
Let $(X,\mu)$ be an analytic Borel space. Let $S$ be a discrete semigroup acting on $X$ by Borel maps. Let $(X \rtimes S,\lambda,\mu)$ be the corresponding measure semigroupoid. Then $(X \rtimes S,\lambda,\mu)$ is Liouville if and only if it is Reiter.
\etheo

For the topological setting, we prove a lemma first.
\blemma
Let $X$ be a locally compact and second countable Hausdorff space, and $S$ a countable discrete semigroup acting on $X$ by continuous maps. Let $(X \rtimes S,\lambda)$ be the corresponding locally compact semigroupoid with quasi-Haar system. The following are equivalent:
\begin{itemize}
  \item[\rm (i)] $(X \rtimes S,\lambda)$ is uniform Reiter,
  \item[\rm (ii)] $(X \rtimes S,\lambda)$ is  uniform and tight Reiter,
  \item[\rm (iii)] For each $\ve > 0$, finite set $F \subseteq S$  and compact set $C \subseteq X$,  there exists a family $\theta = (\theta^x)_{x \in X}$ of probability measures on $S$ satisfying the following conditions:
  \begin{itemize}
       \item[\rm (1)] the map $ x \in X \ma \theta^x\in \ell^1(S)$ is continuous;
       \item[\rm (2)] $\norm{\theta^x - s \theta^{x.s}} < \ve$ for all $s \in F$ and $x \in C$;
       \item[\rm (3)]  $\supp(\theta) \defeq \bigcup_{x \in X} \supp(\theta^x)$ is finite.
  \end{itemize}
\end{itemize}
\elemma
\bproof
Clearly, we have (iii) $\Rarr$ (ii) $\Rarr$ (i). It remains to prove (i) $\Rarr$ (iii). Let $\ve > 0$ with a finite set $F \subseteq S$  and compact set $C \subseteq X$   given in (iii). Since $(X \rtimes S,\lambda)$ is uniform Reiter, there is a family of probability measures $\vartheta = (\vartheta^x)_{x \in X}$ on $S$ such that the map $ x\in X \ma \vartheta^x \in\ell^1(S) $ is continuous and $\norm{\vartheta^x - \delta_s * \vartheta^{x.s}} < \tfrac{\ve}{2}$ for all $s \in F$ and $x \in C$. Since $X$ is locally compact and $\sigma$-compact, there exist compact subsets $K_n$ of $X$ with $C \cup C.F \subseteq K_1$, $K_n \subseteq \accentset{\circ}{K_{n+1}}$ for all $n \in \Nz$ and $X = \bigcup_{n=1}^{\infty} K_n$.

Let us now define probability measures $\vartheta_2^x$ on $S$ for all $x \in K_2$ such that the map $x\in K_2 \ma \vartheta_2^x \in \ell^1(S) $ is continuous. For $1 > \ve' > 0$ and $x \in K_2$, there exists a finite set $E_x \subseteq S$  such that $\norm{ \sum_{s \in E_x} \vartheta^x(s) \delta_s - \vartheta^x } < \ve'$. As
$y\in X \ma \vartheta^y\in \ell^1(S) $ is continuous, the map $y\in X \ma \sum_{s \in E_x} \vartheta^y(s) \delta_s - \vartheta^y\in \ell^1(S)$ is continuous as well. Hence there exists an open neighbourhood $U_x$ of $x$ such that $\norm{ \sum_{s \in F_x} \vartheta^y(s) \delta_s - \vartheta^y } < \ve'$ for all $y \in U_x$. Since $K_2$ is compact, there exist finitely many $x_1, \dotsc, x_N$ in $K_2$ such that $K_2 \subseteq \bigcup_{i=1}^N U_{x_i}$. Define $E \defeq \bigcup_{i=1}^N E_{x_i}$ and for $x \in K_2$, set $\tilde{\vartheta}_2^x \defeq \sum_{s \in E} \vartheta^x(s)$. Obviously, the map  $x\in K_2 \ma \tilde{\vartheta}_2^x\in \ell^1(S)$ is continuous. For all $x \in K_2$, we have $\norm{\tilde{\vartheta}_2^x} \geq \norm{\vartheta^x} - \norm{\tilde{\vartheta}_2^x - \vartheta^x} > 1 - \ve'$. Thus we may form $\vartheta_2^x \defeq \norm{\tilde{\vartheta}_2^x}^{-1} \cdot \tilde{\vartheta}_2^x$ for all $x \in K_2$. By construction, $\vartheta_2^x$ are probability measures on $S$ such that the map $x\in K_2  \ma \vartheta_2^x\in \ell^1(S)$ is continuous. Moreover, for all $x \in K_2$, we have
\bglnoz
  && \norm{\vartheta_2^x - \vartheta^x}
  \leq \norm{\vartheta_2^x - \tilde{\vartheta}_2^x} + \norm{\tilde{\vartheta}_2^x - \vartheta^x}
  = \rukl{\norm{\tilde{\vartheta}_2^x}^{-1} - 1} \norm{\tilde{\vartheta}_2^x} + \norm{\tilde{\vartheta}_2^x - \vartheta^x} \\
  &<& (1 - \ve')^{-1} - 1 + \ve' < \tfrac{\ve}{4}
\eglnoz
for sufficiently small $\ve'$. Hence it follows that for all $s \in E$ and $x \in C$, we have
\bgln
\label{XS_t2-ineq}
  && \norm{\vartheta_2^x - \delta_s * \vartheta_2^{x.s}}
  \leq \norm{\vartheta_2^x - \vartheta^x} + \norm{\vartheta^x - \delta_s * \vartheta^{x.s}} + \norm{\delta_s * \vartheta^{x.s} - \delta_s * \vartheta_2^{x.s}}
  \\
  &<& \tfrac{\ve}{4} + \tfrac{\ve}{2} + \tfrac{\ve}{4} = \ve. \nonumber
\egln
Now define $U_2 \defeq \accentset{\circ}{K_2}$ and $U_n \defeq \accentset{\circ}{K_n} \setminus K_{n-2}$ for all $n \geq 3$.  By construction, we have $\accentset{\circ}{K_n} \setminus \accentset{\circ}{K_{n-1}} \subseteq U_n$ for all $n \geq 3$, and $X \subseteq \bigcup_{n=2}^{\infty} K_{n-1} \subseteq \bigcup_{n=2}^{\infty} \accentset{\circ}{K_n} \subseteq \bigcup_{n=2}^{\infty} U_n$. Since $X$ is locally compact and $\sigma$-compact,  we can find a partition of unity $\{h_n\}_{n\geq 2}$ subordinate to $\{U_n\}_{n\geq 2}$. Define $\theta_2^x \defeq h_2(x) \vartheta_2^x$ for $x \in U_2$ and $\theta_2^x \defeq 0$ for $x \notin U_2$. Then the map $x\in X \ma \theta_2^x\in \ell^1(S)$ is continuous. Now fix $t \in S$ and define, for $n \geq 3$, $\theta_n^x \defeq h_n(x) \delta_t$. Set $\theta^x \defeq \sum_{n=2}^{\infty} \theta_n^x$.

 We claim that $\theta$ has the desired properties. For every $x \in X$, $\norm{\theta_n^x} = h_n(x)$ implies $\norm{\theta^x} = \sum_n h_n(x) = 1$ and $\theta^x$ are probability measures. Moreover, given $x \in X$, there exist $N \geq 2$ such that $x$ lies in $U_N$. But by construction, $U_N \subseteq K_N$ implies  $U_N \cap U_n = \emptyset$ for all $n \geq N+2$. Therefore, $\theta^x = \sum_{n=2}^{N+1} \theta_n^x$ for all $x \in U_N$. This shows continuity of the map  $x\in U_N  \ma \theta^x\in \ell^1(S)$ and $x\in X \ma \theta^x\in \ell^1(S)$. Furthermore, it is clear that for all $x \in X$, we have $\supp(\theta^x) \subseteq E \cup \gekl{t}$. Finally, given $s \in F$ and $x \in C$, both $x$ and $x.s$ lie in $K_1$ and hence $h_n(x) = h_n(x.s) = 0$ for all $n \geq 3$, and  $h_2(x) = h_2(x.s) = 1$. Thus $\theta^x = \theta_2^x = \vartheta_2^x$ and $\theta^{x.s} = \theta_2^{x.s} = \vartheta_2^{x.s}$. We conclude that
\bgloz
  \norm{\theta^x - \delta_s * \theta^{x.s}} = \norm{\vartheta_2^x - \delta_s * \vartheta_2^{x.s}} \overset{\eqref{XS_t2-ineq}}{<} \ve.
\egloz
\eproof
We observe from the definition that the transformation semigroupoid $(X \rtimes S,\lambda)$ in the above lemma is Liouville if and only if it is continuous Liouville, and now, using
 Theorem~\ref{L-->R_t} and Theorem~\ref{R-->L_t}, we can conclude with the following result.
 \btheo\label{87}
Let $X$ be a locally compact and second countable Hausdorff space, and $S$ a countable discrete semigroup acting on $X$ by continuous maps. Let $(X \rtimes S,\lambda)$ be the corresponding locally compact semigroupoid with quasi-Haar system. Then $(X \rtimes S,\lambda)$ is Liouville if and only if it is uniform Reiter.
\etheo

\subsection{Locally compact semigroups}
\label{special-sgp}

Let $S$ be a second countable locally compact semigroup equipped with a positive quasi-invariant Borel measure $\lambda$  such that the map $(s,t)\in S \times S \ma \rukl{d (s \lambda) / d \lambda}(t)$ is Borel. As explained in Remark~\ref{ex_haar}, by identifying $S$ with $\gekl{\rm pt} \rtimes S$, we may view $(S,\lambda)$ as a locally compact semigroupoid with a quasi-Haar system. In this case, $(S,\lambda)$ is Reiter, in the sense of Definition~\ref{def-Reiter_t}, if  there exists a sequence $(\theta_n)$ of Borel probability measures on $S$ with $\theta_n \prec \lambda$ for all $n \in \Nz$, and $\lim_{n \to \infty} \norm{\theta_n - s \theta_n} = 0$ for all $s \in S$.

\bremark
\label{RA_meas-sgp}
In the above setting, $(S,\lambda)$ is Reiter if and only if there exists a left invariant mean on $L^{\infty}(S,\lambda)$ since the proof in \cite[\S\,4.2]{Li} carries over. Therefore there is no difference between Reiter's condition and amenability for $(S,\lambda)$. 
\eremark

 To describe the Liouville property for $(S,\lambda)$ according to Definition~\ref{def-L_t}, let $\pi$ be a probability measure  on $S$ with $\pi \prec \lambda$. In this case, the fibrewise Markov operators reduce  to a single one $P: \: L^{\infty}(S,\lambda) \to L^{\infty}(S,\lambda)$ with $P(f)(s) = \int_S f(st) d \pi(t)$, and  $H_{\pi}(S) \defeq \menge{f \in L^{\infty}(S,\lambda)}{P(f) = f}$ is the space of $\pi$-harmonic functions in
 $L^\infty(S,\lambda)$. The locally compact semigroup  $(S,\lambda)$ is Liouville if there exists a probability measure $\pi$ on $S$ with $\pi \prec \lambda$ such that $H_{\pi}(S) = \Cz \mathbf 1$.

\btheo
\label{special_thm_meas-sgp}
In the situation described above, $(S,\lambda)$ is Liouville if and only if it is Reiter.
\etheo
\bproof
The key observation is that, by definition, $(S,\lambda)$ is Liouville if and only if $(\gekl{\rm pt} \rtimes S,\delta \times \lambda,\delta)$ is Liouville as a measure semigroupoid (where $\delta$ is the point mass). By Theorem~\ref{L-->R_m}, we know that $(S,\lambda)$ is Reiter if it is Liouville. For the converse, let $(S,\lambda)$ be Reiter. This means that $(S,\lambda)$ is Reiter as a locally compact semigroupoid with quasi-Haar system. Again by definition, this implies that $(\gekl{\rm pt} \rtimes S,\delta \times \lambda,\delta)$ is Reiter as a measure semigroupoid. Hence Theorem~\ref{R-->L_m} implies that $(\gekl{\rm pt} \rtimes S,\delta \times \lambda,\delta)$ is Liouville as a measure semigroupoid. Hence the locally compact semigroup $(S,\lambda)$ is Liouville.
\eproof

The question remains  which semigroups $S$ admit a quasi-invariant measure $\lambda$ with a Borel
map $(s,t)\in S \times S \ma \rukl{d (s \lambda) / d \lambda}(t)$. This is for instance the case if $S$ is a locally compact subsemigroup of a second countable locally compact group $G$, that is, $S$ is a subsemigroup of $G$ as well as a locally compact subspace. If the Haar measure $\lambda_G$ on $G$ satisfies $\lambda_G(S) \neq 0$, then the restriction $\lambda$ of $\lambda_G$ to $S$ is a measure with the desired properties.

Of course, another class of examples is given by countable discrete semigroups $S$ with identity and counting measure $\lambda$. For these semigroups, we have the following result from Remark~\ref{RA_meas-sgp} and Theorem~\ref{special_thm_meas-sgp}. In the next subsection, we discuss the case of metric semigroups, which need not be locally compact nor support a quasi-invariant measure.
\btheo
\label{610}
Let $S$ be a countable discrete semigroup with identity. The following conditions are equivalent.
\begin{itemize}
\item[\rm (i)] $S$ is Liouville, that is, there exists a probability measure $\pi$ on $S$ such that every bounded $\pi$-harmonic function  is constant;
\item[\rm (ii)] $S$ is Reiter, that is, there exists a sequence $(\theta_n)_n$ of probability measures on $S$ such that $\lim_{n \to \infty} \norm{\theta_n - s \theta_n} = 0$ for all $s \in S$;
\item[\rm (iii)] $S$ is amenable.
\end{itemize}
\etheo

In the last section, we give some examples of discrete semigroups with the Liouville property for all non-degenerate probability measures, i.e., with the property that every bounded $\pi$-harmonic function is constant for every non-degenerate probability measure $\pi$.

\subsection{Metric semigroups}
\label{sgp}

Following the previous results for semigroups equipped with a quasi-invariant measure, it is natural to enquire about semigroups without such a measure.  We present the case of metric semigroups in this subsection, but refer to \cite{CLi} for further discussion of topological semigroups.

Given a probability measure $\pi$ on a metric semigroup $S$, we denote by
$H_\pi(S)$
the closed subspace of $LUC(S)$ consisting of $\pi$-harmonic functions on $S$.
 We say that
$S$ has the {\it Liouville property} if there is a tight probability measure $\pi$ on $S$ such that $H_{\pi}(S) = \mathbb{C}\mathbf{1}$ where $\mathbf 1$ denotes the constant function on $S$ with value $1$.

A metric semigroup $S$ is said to satisfy {\it Reiter's condition} if for every $\ve > 0$ and compact set $K \subseteq S$, there is a probability measure $\theta$ on $S$ with compact support such that $\norm{\theta - \delta_s * \theta} < \ve$ for all $s \in K$.

For a countable discrete semigroup $S$,  the notion of the Liouville property  and Reiter's condition just introduced agrees with the one previously defined.

Given a $\sigma$-compact metric semigroup $S$ satisfying Reiter's condition, a construction analogous to that in the proof of Proposition \ref{tgpd-prop-am-->L} yields a tight probability measure $\pi$ on $S$ such that
\bgl\label{pin}
\lim_{n \rightarrow \infty} \|\delta_s * \pi^n - \pi^n\| =0
\egl
for all $s \in S$. In fact, this is also true for topological semigroups and a detailed proof will be given in \cite{CLi}.
Consequently, we have the following result.

\begin{theorem}
\label{R-L_thm_msgp}
Let $S$ be a  $\sigma$-compact metric semigroup with identity $e$ satisfying Reiter's condition. Then $S$ enjoys the Liouville property.
\end{theorem}
\begin{proof} By the previous observation, there is a tight probability measure $\pi$ on $S$ satisfying (\ref{pin})
for all $s \in S$.
Hence for each $\pi$-harmonic function $f\in LUC(S)$, we have
\begin{eqnarray*}
|f(x)- f(e)|&=|\int_S (f(xy)-f(y))d\pi^n(y)|
=|\int_S(\delta_x * f -f)d\pi^n| \hspace{1.1in}\\
&=|\int_S f d(\delta_x * \pi^n -\pi^n)| \leq\|f\|_{\infty} \|\delta_x * \pi^n - \pi^n\|
 \rightarrow 0\quad {\rm as}\quad n \rightarrow \infty
\end{eqnarray*}
and it follows that $f$ is constant. This completes the proof.
\end{proof}

We conclude this section by showing that a metric semigroup $S$ with the Liouville property, but not necessarily $\sigma$-compact,
must be amenable. This will follow readily from the fact that
 the space $H_\pi(S)$ of bounded $\pi$-harmonic functions on  $S$ is the range of a contractive projection on $LUC(S)$, which commutes with left translations.   To show this, we need the following two lemmas which are straightforward extension of similar results in \cite{CLau} for groups.

 Let $\rho$ denote the topology of pointwise convergence in $LUC(S)$ and let $\tau$ be the topology of uniform convergence on compact sets in $S$.

\begin{lemma}
\label{Kf-cpct}
Let $S$ be a metric semigroup and let $f \in LUC(S)$. Denote by $K_f$ the
$\rho$-closure $\overline{\rm co}\,^{\rho} \{f*\delta_s: s\in S\}$ of the convex hull of $\{f*\delta_s: s\in S\}$. Then the topologies $\rho$ and $\tau$ coincide on $K_f$ which is compact in these topologies.
\end{lemma}

\begin{lemma}\label{in} Let $f\in LUC(S)$ and $\pi\in M(S)$ be a probability measure. Then $f* \pi \in K_f$.
\end{lemma}

\begin{proposition}\label{p} Let $S$ be a  metric semigroup and let $\pi\in M_t(S)$ be a
probability measure. Then there exists a contractive projection $P: \: LUC(S)\to LUC(S)$ with range equal to $H_{\pi}(S).$ Moreover, $P$ commutes with left translations on $S$.
\end{proposition}
\begin{proof}
 The arguments are similar to those given in \cite{Chu} and \cite{CLau} for groups, but we include the proof for completeness. Let $LUC(S)$ be equipped with the topology $\rho$ of pointwise convergence  and let the Cartesian product $LUC(S)^{LUC(S)}$ be equipped with the product topology. Define a linear map $L: LUC(S) \to LUC(S)$ by $L(f) = f* \pi \qquad (f\in LUC(S))$.

Consider $L$ as an element in $LUC(S)^{LUC(S)}$, and so are the $n$-th iterates
$L^n = \overbrace{L\circ \cdots\circ L}^{\rm n-times}$.  By Lemma \ref{in}, $L^n \in  \prod_{f \in LUC(S)} K_f \subseteq LUC(S)^{LUC(S)}$ for all $n$. It follows that
the closed convex hull $K:=\overline{ \rm co}\,\{L^n : n =1,2, \ldots \} \subseteq \prod_f K_f$ is compact in the product topology by Lemma \ref{Kf-cpct}.

 Define an affine map $T: K \rightarrow K$ by $T(\Lambda)(f)= \Lambda(f) * \pi \qquad (\Lambda \in K, f \in LUC(S))$.
Then $T$ is continuous. Indeed, given a net $(\Lambda_\alpha)_{\alpha}$ in $K$ converging to $\Lambda \in K$, then $\Lambda_\alpha(f) \longrightarrow \Lambda(f)$ pointwise on $S$
for each $f\ \in LUC(S)$. By Lemma \ref{in}, $\Lambda(f), \Lambda_\alpha(f) \in K_f$ and hence Lemma
\ref{Kf-cpct} implies that $(\Lambda_\alpha(f))_{\alpha}$ converges to $\Lambda(f)$ uniformly on compact sets in $S$. It
follows  that $(\Lambda_\alpha(f)* \pi)_{\alpha}$ converges pointwise to $\Lambda(f)*\pi$. This proves that $(T(\Lambda_\alpha))_{\alpha}$ converges to $T(\Lambda)$ in the product topology.

By the Markov-Kakutani fixed-point theorem, $T$ admits a fixed point $P \in K$ which gives
$$P(f) = P(f) * \pi \qquad (f\in LUC(S)).$$
It is easy to see that $P^2=P$ and $P(LUC(S)) = H_\pi(S)$. The last assertion follows from the fact that left translations on $S$ commute with the operator $L$.
\end{proof}

\begin{theorem}\label{1} Let $S$ be a metric semigroup with the Liouville property.  Then $S$ is amenable.
\end{theorem}
\begin{proof} Let $\pi$ be a tight
probability measure on $S$ such that $H_\pi(S) = \mathbb{C}\mathbf 1$.  Let $P: LUC(S) \rightarrow LUC(S)$ be the contractive projection constructed in Proposition \ref{p}. Since $P(LUC(S)) = H_\pi(S)$, we have $P(f) = \vp(f)\mathbf 1 \qquad (f \in LUC(S))$ for a unique functional $\vp : LUC(S) \rightarrow \mathbb{C}$. Evidently, $\vp$ is a left invariant mean.
\end{proof}

The above approach using the contractive projection $P$ gives an alternative
 proof of the implication (i) $\Rightarrow$ (iii) in Theorem \ref{610}.

\section{Examples}
\label{nilp}

A measure on a semigroup is called {\it non-degenerate} if its support generates the semigroup.
It is known that if a locally compact abelian semigroup $S$ supports a  non-degenerate probability measure $\pi$, then $H_\pi(S)= \mathbb{C}\mathbf 1$ (cf.\,\cite{lz,rs}).
We now present a class of discrete non-abelian semigroups with the Liouville property for all non-degenerate probability measures.

\bdefin \label{nil}\rm
Let $S$ be a discrete semigroup with identity $e$. A
{\it central series} of $S$ is a finite chain
$$
  \gekl{e} = S_0 \subsetneq S_1 \subsetneq \dotso \subsetneq S_n = S,
$$
of subsemigroups in $S$ such that
\begin{enumerate}
\item[(i)] $S_m$ is {\it right reversible} for all $0 \leq m \leq n$, that is, for all $x, y$ in $S_m$, we have $S_m x \cap S_m y \neq \emptyset$;
\item[(ii)] For every $1 \leq m \leq n$, $a \in S_m$ and $s \in S$, there exist $x, y$ in $S_{m-1}$ with $xas = ysa$.
\end{enumerate}
We call $n$ the {\it length} of the above central series.
\edefin
Semigroups admitting central series are a natural generalisation of nilpotent groups. We should note, however, that the concept of nilpotent semigroups is not defined in terms of the  existence of central series (cf. \cite{Mal}). Evidently, every abelian semigroup admits a central series of length 1.

\bexample
The Heisenberg semigroup
$$
  S = \menge{
  \rukl{
  \begin{smallmatrix}
  1 & x & z \\
  0 & 1 & y \\
  0 & 0 & 1
  \end{smallmatrix}
  }
  \in GL_3(\Rz)}{x,y,z \geq 0}
$$
admits a central series of length $2$.

To see this, set
$$
  S_0 = \gekl{
  \rukl{
  \begin{smallmatrix}
  1 & 0 & 0 \\
  0 & 1 & 0 \\
  0 & 0 & 1
  \end{smallmatrix}
  }
  },
  S_1 = \menge{
  \rukl{
  \begin{smallmatrix}
  1 & 0 & z \\
  0 & 1 & 0 \\
  0 & 0 & 1
  \end{smallmatrix}
  }
  }{z \geq 0}
  \ and \
  S_2 = S.
$$
Condition (i) is obviously satisfied for $S_0$ and $S_1$ ($S_1$ is abelian), and it holds for $S_2 = S$ as well because of
$$
  \rukl{
  \begin{smallmatrix}
  1 & 0 & \zeta + x \eta \\
  0 & 1 & 0 \\
  0 & 0 & 1
  \end{smallmatrix}
  }
  \rukl{
  \begin{smallmatrix}
  1 & \xi & 0 \\
  0 & 1 & \eta \\
  0 & 0 & 1
  \end{smallmatrix}
  }
  \rukl{
  \begin{smallmatrix}
  1 & x & z \\
  0 & 1 & y \\
  0 & 0 & 1
  \end{smallmatrix}
  }
  =
  \rukl{
  \begin{smallmatrix}
  1 & 0 & z + \xi y \\
  0 & 1 & 0 \\
  0 & 0 & 1
  \end{smallmatrix}
  }
  \rukl{
  \begin{smallmatrix}
  1 & x & 0 \\
  0 & 1 & y \\
  0 & 0 & 1
  \end{smallmatrix}
  }
  \rukl{
  \begin{smallmatrix}
  1 & \xi & \zeta \\
  0 & 1 & \eta \\
  0 & 0 & 1
  \end{smallmatrix}
  }
$$
for all $x, y, z$ and $\xi, \eta, \zeta$.

Condition (ii) is true for $m=1$ since $S_1$ is abelian. For $m=2$, condition (ii) holds because we have for all $x, y, z$ and $\xi, \eta, \zeta$:
$$
  \rukl{
  \begin{smallmatrix}
  1 & 0 & x \eta \\
  0 & 1 & 0 \\
  0 & 0 & 1
  \end{smallmatrix}
  }
  \rukl{
  \begin{smallmatrix}
  1 & \xi & \zeta \\
  0 & 1 & \eta \\
  0 & 0 & 1
  \end{smallmatrix}
  }
  \rukl{
  \begin{smallmatrix}
  1 & x & z \\
  0 & 1 & y \\
  0 & 0 & 1
  \end{smallmatrix}
  }
  =
  \rukl{
  \begin{smallmatrix}
  1 & 0 & \xi y \\
  0 & 1 & 0 \\
  0 & 0 & 1
  \end{smallmatrix}
  }
  \rukl{
  \begin{smallmatrix}
  1 & x & z \\
  0 & 1 & y \\
  0 & 0 & 1
  \end{smallmatrix}
  }
  \rukl{
  \begin{smallmatrix}
  1 & \xi & \zeta \\
  0 & 1 & \eta \\
  0 & 0 & 1
  \end{smallmatrix}
  }
  .
$$
Actually, exactly the same computations show that for every subring $R$ of $\Rz$, the semigroup
$$
  S = \menge{
  \rukl{
  \begin{smallmatrix}
  1 & x & z \\
  0 & 1 & y \\
  0 & 0 & 1
  \end{smallmatrix}
  }
  \in GL_3(\Rz)}{x,y,z \in R \cap [0,\infty)}
$$
admits a central series of length $2$.
\eexample

We show that semigroups with central series have the Liouville property for all non-degenerate probability measures. This generalizes the corresponding result for
nilpotent groups \cite{DM,Mar} (see also \cite{CH}). First, we sketch a proof of the following lemma, using similar arguments in \cite{CH}.
\blemma
\label{harmonic-center}
Let $S$ be a semigroup with identity $e$ and  $\pi$ a non-degenerate probability measure on $S$. Then for each $f \in H_{\pi}(S)$, we have $f(ax) = f(x)$ for all $x \in S$ and $a$ in the center of $S$.
\elemma
\bproof
Considering the real and imaginary parts  separately,
we may assume that $f$ is real-valued.
Fix $a$ in the center of $S$. Set $g(x) = f(x) - f(ax)$, where $g \in \ell^{\infty}(S)$. Let $\alpha \defeq \sup_{x \in S} g(x)$. Then $g \in H_{\pi}(S)$. Choose a sequence $(x_n)_n$ in $S$ with $\alpha = \lim_{n \to \infty} g(x_n)$. Define $g_n \in \ell^{\infty}(S)$ by  $g_n(x) \defeq g(x_n x)$. Then $g_n \in H_{\pi}(S)$ and $\lim_{n \to \infty} g_n(e) = \alpha$, where $g_n(x) = g(x_nx) \leq \alpha$ for all $x \in S$. We also know that $\norm{g_n}_{\infty} \leq \norm{g}_{\infty}$.

Let $h$ be a w*-limit point of $g_n$ in $\ell^{\infty}(S)$. Since $g_n \in H_{\pi}(S)$ and the map $\vp\in \ell^{\infty}(S) \mapsto  \vp * \pi \in \ell^{\infty}(S)$ is w*-continuous, we have $h=h * \pi$ in $\ell^\infty(S)$ and therefore $h \in H_{\pi}(S)$. As w*-convergence in $\ell^\infty(S)$ implies pointwise convergence, $g_n(x) \leq \alpha$  implies $h(x) \leq \alpha$ for all $x \in S$. Since $\lim_{n \to \infty} g_n(e) = \alpha$, we have $h(e) = \alpha$.

We show $h(a) = \alpha$. Since $S = \bigcup_{n=1}^{\infty} ({\rm supp}\,\pi)^n = \bigcup_{n=1}^{\infty} \supp\,\pi^{n}$, one can choose $n \in \Nz$ such that $\pi^{n}(a) > 0$.  If $h(a) < \alpha$, then
\bglnoz
  && \alpha = h(e) = h*\pi^n(e)= \sum_{y \in S} h(y) \pi^{n}(y) = h(a) \pi^{n}(a) + \sum_{y \neq a} h(y) \pi^{n}(y) \\
  &<& \alpha \pi^{n}(a) + \sum_{y \neq a} \alpha \pi^{n}(y) = \alpha
\eglnoz
which is a constradiction. Hence we must have $h(a) = \alpha$. Similarly, $h(a^p) = \alpha$ for all $p \in \Nz$.

Now if $\alpha > 0$, then we can choose $m \in \Nz$ with $\norm{f}_{\infty} \leq \tfrac{1}{4} m \alpha$, and $n \in \Nz$ with $g_n(a^p) > \tfrac{1}{2} \alpha$ for all $1 \leq p \leq m$. This gives
\bglnoz
  && \tfrac{1}{2} m \alpha < \sum_{p=1}^{m} g_n(a^p) = \sum_{p=1}^{m} f(x_n a^p) - f(a x_n a^p)
  = \sum_{p=1}^{m} f(x_n a^p) - f(x_n a^{p+1}) \\
  &=& f(x_n a) - f(x_n a^{p+1}) \leq 2 \cdot \tfrac{1}{4} m \alpha = \tfrac{1}{2} m \alpha
\eglnoz
 which is impossible. Therefore $\alpha \leq 0$ and we have $g(x) \leq 0$ for all $x \in S$, that is, $f(x) \leq f(ax)$ for $x \in S$. Repeating the previous argument for the function $ f(ax) - f(x)$, we get $f(ax) \leq f(x)$   and hence    $f(x) = f(ax)$ for all $x \in S$.
\eproof

\btheo
\label{CS-->Liouville}
Let $S$ be a semigroup with identity and a central series. Then for every non-degenerate probability measure $\pi$ on $S$, we have $H_{\pi}(S) = \Cz \mathbf{1}$.
\etheo

\bproof
We proceed inductively on the length $n$ of the central series of $S$. The case $n=0$ is trivial. For the inductive step, assume the assertion is true for length $n$ and let $$\gekl{e} = S_0 \subsetneq S_1 \subsetneq \dotso \subsetneq S_n \subseteq S_{n+1} = S$$ be a central series of length $n+1$. Since $S$ has an identity in $S_1$ which is right reversible,
we can define an equivalence relation $\sim$ on $S$ by setting $x \sim y$ if there exist $a,b \in S_1$ satisfying $ax=by$.  The set $S / {}_{\sim}$ of equivalence classes  $\dot{x}$ is a semigroup structure with product $\dot{x} \cdot \dot{y} \defeq (xy)\dot{ }$, which is well-defined because $S_1$ is in the center of $S$.Observe that $$\gekl{\dot{e}} = S_1 / {}_{\sim} \subsetneq S_2 / {}_{\sim} \subsetneq \dotso \subsetneq S_n / {}_{\sim} \subseteq S_{n+1} / {}_{\sim} = S / {}_{\sim}$$ is a central series of length at most $n$ for $S / {}_{\sim}$.

Now let $\pi$ be a probability measure on $S$ such that $\supp\,\pi$ generates $S$ as a semigroup. Let $q: \ S \to S / {}_{\sim}$ be the quotient map. Define a probability measure $\dot{\pi}$ on $S / {}_{\sim}$ by setting $\dot{\pi}(\dot{x}) \defeq \sum_{x \in q^{-1}(\dot{x})} \pi(x)$. Then $\supp\,\dot{\pi}$ generates $S / {}_{\sim}$ as a semigroup. By the inductive hypothesis, we have $H_{\dot{\pi}}(S / {}_{\sim}) = \Cz \mathbf{1}$.

Let $f \in H_{\pi}(S)$. Define $\dot{f}: \ S / {}_{\sim} \to \Cz$ by  $\dot{f}(\dot{x}) \defeq f(x)$. This is well-defined. Indeed, if $x \sim y$, then there exist $a,b \in S_1$ with $ax=by$. As $S_1$ is contained in the center of $S$, Lemma~\ref{harmonic-center} yields  $f(x) = f(ax) = f(by) = f(y)$. One verifies readily that $\dot{f} * \dot{\pi}=\dot{f}$.
 Therefore $\dot{f}$, and hence  $f$, must be constant.
\eproof

\end{document}